\documentclass[a4]{amsart}

\usepackage{amssymb}
\usepackage{mathtools}
\usepackage[utf8]{inputenc}
\usepackage[english]{babel}
\usepackage{enumitem}
\setenumerate[1]{label=(\alph*)}
\setenumerate[2]{label=(\roman*)}
\usepackage[all]{xy}
\usepackage{tikz}
\usetikzlibrary{shapes.geometric,matrix,arrows}
\usepackage{upgreek}
\usepackage{leftidx}
\usepackage[normalem]{ulem}
\usepackage{hyperref}

\numberwithin{equation}{section}
 
\newtheorem{mtheorem}{Theorem}
\newtheorem{theorem}{Theorem}[section]
\newtheorem{proposition}[theorem]{Proposition}
\newtheorem{lemma}[theorem]{Lemma}
\newtheorem{corollary}[theorem]{Corollary}
\newtheorem{setting}[theorem]{Setting}

\theoremstyle{definition}
\newtheorem{definition}[theorem]{Definition}
\newtheorem{example}[theorem]{Example}

\theoremstyle{remark}
\newtheorem{remark}[theorem]{Remark}

\newcommand{\define}[1]{{\emph{#1}}}
\newcommand{\sptag}[1]{\href{http://stacks.math.columbia.edu/tag/#1}{Tag~#1}}

\renewcommand{\epsilon}{{\varepsilon}}

\renewcommand{\phi}{{\varphi}}

\DeclareMathOperator{\im}{Im}

\DeclareMathOperator{\Hom}{Hom}

\DeclareMathOperator{\HH}{H}

\DeclareMathOperator{\Spec}{Spec}

\newcommand{\sheafHom}{{\mathcal{H}{om}}}

\newcommand{\liset}{{\operatorname{lis-\acute{e}t}}}

\newcommand{\Mod}{{\ensuremath{\operatorname{Mod}}}}
\newcommand{\id}{{\operatorname{id}}}

\newcommand{\qc}{{\operatorname{qc}}}
\newcommand{\pc}{{\operatorname{pc}}}
\newcommand{\co}{{\operatorname{coh}}}
\newcommand{\pf}{{\operatorname{pf}}}
\newcommand{\sg}{{\operatorname{sg}}}
\newcommand{\D}{{\operatorname{D}}}

\newcommand{\bd}{{\operatorname{b}}}
\newcommand{\locbd}{{\operatorname{lb}}}
\newcommand{\Rd}{\mathsf{R}}

\newcommand{\ol}[1]{{\overline{#1}}}
\newcommand{\ra}{\rightarrow}

\newcommand{\rla}{\rightleftarrows}
\newcommand{\xra}[2][]{\xrightarrow[{#1}]{#2}}
\newcommand{\xla}[2][]{\xleftarrow[{#1}]{#2}}
\newcommand{\sira}{\xra{\sim}}

\newcommand{\Ra}{\Rightarrow}
\newcommand{\xRa}[2][]{\xRightarrow[{#1}]{#2}}
\newcommand{\siRa}{\xRa{\sim}}

\hypersetup{  
  pdftitle    = {Conservative descent for semi-orthogonal decompositions},
  pdfauthor   = {Daniel Bergh, Olaf M.~Schn{\"u}rer},
  colorlinks=true}

\title[Conservative descent]{Conservative descent for semi-orthogonal decompositions}
\begin{document}

\author{Daniel Bergh \and Olaf M.~Schn{\"u}rer}

\address{       
  Department of Mathematical Sciences\\
  Copenhagen University\\
  Universitetsparken~5\\
  2100~København~Ø\\ 
  Denmark
}
\email{dbergh@math.ku.dk}

\address{
  Institut f\"ur Mathematik\\
  Universit{\"a}t Paderborn\\
  Warburger Stra\ss{}e 100\\
  33098 Paderborn\\
  Germany
}

\email{olaf.schnuerer@math.uni-paderborn.de}


\begin{abstract}
  Motivated by the local flavor of several well-known
  semi-orthogonal decompositions in algebraic geometry, we
  introduce a technique called 
  \emph{conservative descent}, which 
  shows that it is enough to establish these decompositions
  locally. The decompositions we have in mind are those for
  projectivized vector bundles and blow-ups, due to Orlov, and
  root stacks, due to Ishii and Ueda.
  Our technique simplifies the proofs of these decompositions and
  establishes them in greater generality for arbitrary algebraic
  stacks.
\end{abstract}


\maketitle
\tableofcontents

\section{Introduction}
\label{sec:introduction}
Semi-orthogonal decompositions (see Definition~\ref{d:semi-orth})
of derived categories are central in the study of
non-commutative aspects of algebraic geometry.
Such decompositions are well-known for derived categories of
projectivized vector bundles, blow-ups and root stacks.
More examples are given in a survey by Kuznetsov~\cite{kuznetsovICM}.
However, the references describing these
decompositions usually 
impose quite restrictive conditions on the geometric objects
under consideration (cf.~Remark~\ref{rem-references}).
For instance, they often only consider smooth and projective varieties over a field. 
This has been too limited for our purposes,
and this article grew out of the need to generalize the semi-orthogonal decompositions
mentioned above to the context of algebraic stacks.
We use the results of the present article in our recent work \cite{bls2016},
where we show that the derived category of
a smooth, proper Deligne--Mumford stack is a semi-orthogonal summand of the derived category of a smooth projective variety.
We also use them in the follow-up article \cite{bs2019},
where we generalize a result by Bernardara \cite{bernardara-BS-schemes} on semi-orthogonal decompositions
of derived categories of Brauer--Severi varieties using stacky methods.
Furthermore, the results are used in \cite{bgll2017} to define the categorical measure for Deligne--Mumford stacks.

Naively, one would expect that most statements about schemes and
varieties generalize via simple descent arguments to algebraic stacks.
However, when working with derived categories,
it is not clear that such an approach works since
derived categories do not satisfy descent in the usual sense.
The problem is that the derived pull-back along a faithfully flat morphism is usually
not faithful.
Indeed, this can already be seen by considering the covering of the projective
line by its standard affine charts.

A modern approach to overcome such obstacles is to
consider enhancements of derived categories,
either by differential graded categories or by $\infty$-categories.
However, this comes at the cost of invoking a substantial amount of technical machinery.

Our observation is that there is a technically less demanding way to solve the problem.
It is based on the fact that even though the derived pull-back along a faithfully flat
morphism is not faithful, it has the weaker property of being \emph{conservative}.
That is, the functor reflects isomorphisms.
We formalize a technique, which we call \emph{conservative descent},
which allows us to descend certain semi-orthogonal decompositions along
conservative triangulated functors.
Our methods are described entirely in the classical language of triangulated categories.

In the abstract setting of triangulated categories, our conservative descent theorem is formulated as follows.
\begin{mtheorem}[{see Theorem~\ref{t-main-categorical}}]
\label{t-main-categorical-intro}
Let $\mathcal{T}$ and~$\mathcal{T}'$ be triangulated categories
and let $F \colon \mathcal{T} \ra \mathcal{T}'$ be a conservative 
triangulated functor.
Let $S_1, \dots, S_n$ and $S_1', \dots, S_n'$ be sequences of
idempotent comonads (see Definition~\ref{d:comonad}) on $\mathcal{T}$ and $\mathcal{T}'$,
respectively,
and assume that $S_i$ and $S_i'$ are compatible with respect to
$F$ for each~$i$ 
in the sense of Definition~\ref{d:compatibility}.
If the sequence 
\begin{equation}
\label{eq:sod-upstairs-intro}
\im S'_1, \ldots, \im S'_n
\end{equation}
of essential images is semi-orthogonal in $\mathcal{T}'$,
then so is the sequence 
\begin{equation}
\label{eq:sod-downstairs-intro}
\im S_1, \ldots, \im S_n
\end{equation}
in $\mathcal{T}$. Both sequences consist of right admissible
subcategories of $\mathcal{T}'$ and $\mathcal{T}$, respectively.
Moreover, if \eqref{eq:sod-upstairs-intro} is a semi-orthogonal decomposition of $\mathcal{T}'$,
then \eqref{eq:sod-downstairs-intro} is a semi-orthogonal decomposition of $\mathcal{T}$.
\end{mtheorem}

The important point here is compatibility between the comonads $S_i$ and~$S_i'$.
In order to formulate a statement which is easy to apply in geometric situations,
we introduce the notion of \emph{relative Fourier--Mukai transforms} (see~Definition~\ref{def-fourier-mukai}).
Just as their classical counterparts, these are functors attached
to certain geometric data.
The comonads in the statement of Theorem~\ref{t-main-categorical-intro}
are typically induced by relative Fourier--Mukai transforms.
Due to the relative flavor of our definition,
these transforms admit a natural notion of flat base change,
and the induced comonads automatically satisfy the required compatibility conditions.
This allows us to give the following formulation of the conservative descent theorem in the geometric 
setting.

\begin{mtheorem}[{see Theorem~\ref{t:main-geometric}}]
\label{t:main-geometric-intro}
Let $Z_1, \ldots, Z_n$ and~$X$ be algebraic stacks over some base algebraic stack $S$,
and assume that $\Phi_i \colon \D_\qc(Z_i) \to \D_\qc(X)$, for $1 \leq i \leq n$,
is a Fourier--Mukai transform over $S$ in the sense of Definition~\ref{def-fourier-mukai}.
Let $u\colon S' \to S$ be a faithfully flat morphism,
and denote the base change of the objects above by $Z'_1, \ldots, Z'_n$, $X'$ and $\Phi'_i \colon \D_\qc(Z'_i) \to \D_\qc(X')$, respectively.
Then for each $i$, the functor $\Phi_i$ is fully faithful
provided that $\Phi'_i$ is fully faithful.

Assume that all $\Phi_i'$, and therefore also all $\Phi_i$, are fully faithful.
If the sequence  
\begin{equation}
\label{eq-sod-local-intro}
\im \Phi'_1, \ldots, \im \Phi'_n
\end{equation}
of essential images is semi-orthogonal in $\D_\qc(X')$,
then so is the sequence
\begin{equation}
\label{eq-sod-intro}
\im \Phi_1, \ldots, \im \Phi_n
\end{equation}
in $\D_\qc(X)$.
Both sequences consist of right admissible
subcategories of $\D_\qc(X')$ and $\D_\qc(X)$, respectively.
Moreover, if \eqref{eq-sod-local-intro} is a semi-orthogonal decomposition of $\D_\qc(X')$,
then \eqref{eq-sod-intro} is a semi-orthogonal decomposition of $\D_\qc(X)$.
\end{mtheorem}

As applications of Theorem~\ref{t:main-geometric-intro}, we obtain generalizations and unified proofs of
the semi-orthogonal decompositions for projectivized vector bundles, blow-ups and root stacks.
We summarize the results in the following theorem and refer to Section~\ref{sec:applications} for precise statements.
\begin{mtheorem}
\label{t:applications}
The unbounded derived category $\D_\qc(X)$ of an algebraic stack $X$
admits a naturally defined semi-orthogonal decomposition 
\begin{enumerate}[label=(\roman*)]
\item
\label{enum:projective-bundle}
if $X$ is
a projectivized vector bundle $\mathbb{P}_S(\mathcal{E})$, where $S$ is an algebraic stack and 
$\mathcal{E}$ is a locally free $\mathcal{O}_S$-module of constant finite rank
(Theorem~\ref{t:sod-projective-bundle}, Corollary~\ref{c:sod-projective-bundle});
\item
\label{enum:blow-up}
if $X$ is
a blow-up $\widetilde{Y}$ of an algebraic stack $Y$ 
in a regular closed immersion $Z \hookrightarrow Y$ of constant codimension $c$
(Theorem~\ref{t:sod-blow-up}, Corollary~\ref{c:sod-blow-up});
\item
\label{enum:root-stack}
if $X$ is
an $r$-th root stack $\widetilde{Y}$ of an algebraic stack $Y$ 
in an effective Cartier divisor
$E \subset Y$
(Theorem~\ref{t:sod-root-stack}, Corollary~\ref{c:sod-root-stack}).
\end{enumerate}
Moreover, there are induced semi-orthogonal decompositions of the subcategories $\D_\pf(X)$
of perfect complexes and $\D_\pc^\locbd(X)$ of locally bounded pseudo-coherent complexes
as well as of the singularity category $\D_\sg(X)$, which is defined as the Verdier quotient
$\D_\pc^\locbd(X)/\D_\pf(X)$.
Note that the category $\D_\pc^\locbd(X)$ coincides with the category $\D^\bd_\co(X)$ of complexes with bounded coherent cohomology if $X$ is noetherian.
\end{mtheorem}

\begin{remark}
\label{rem-references}
The semi-orthogonal decompositions mentioned in Theorem~\ref{t:applications} are well known
in several special cases.
Here we give the relevant references.
\begin{enumerate}
\item Be\u{\i}linson proves a version of part
    \ref{enum:projective-bundle} for $\mathbb{P}^n_k$ where $k$ is a
    field in \cite{beilinson1978}. His proof generalizes to
    arbitrary algebraic stacks $S$.
\item
Orlov proved versions of part
\ref{enum:projective-bundle} and~\ref{enum:blow-up} of Theorem~\ref{t:applications} in
\cite[Theorem~2.6 and Theorem~4.3]{orlov1992}.
He works in the context of smooth projective varieties over the field of complex numbers
and states the result as a semi-orthogonal decomposition of $\D^\bd_\co(X)$.
It is not clear to us how to generalize his proof of
\ref{enum:blow-up} to arbitrary algebraic stacks $Y$.
A proof of \ref{enum:blow-up}, basically in the same setting as Orlov's version,
also appears in the textbook by Huybrechts~\cite[Proposition~11.18]{huybrechts2006}.
Also this proof seems to be hard to generalize.
\item
Elagin generalizes Orlov's versions of part
\ref{enum:projective-bundle} and~\ref{enum:blow-up} of Theorem~\ref{t:applications}
to stack quotients by linearly reductive group schemes
\cite[Theorem~10.1, Theorem~10.2]{elagin2012}.
He uses a cohomological descent argument where he works locally on the source rather than locally on the base as we do.
\item
A version of part \ref{enum:root-stack} of
Theorem~\ref{t:applications} is given by Ishii--Ueda
\cite[Proposition~6.1]{iu2015}.
We give a more detailed statement in \cite[Theorem~4.7]{bls2016}.
\end{enumerate}
\end{remark}

\begin{remark}
\label{rem:thomason}
Thomason provides in \cite{thomason-fibre-projectif} a proof
of the Grothendieck--Berthelot--Quillen theorem which describes the
algebraic K-theory of a projectivized vector bundle as a product
of copies of the algebraic K-theory of the base scheme; in
\cite{thomason1993} he describes the algebraic K-theory of a
blow-up in a similar way. 
These two descriptions follow from the
semi-orthogonal decompositions of the subcategories of perfect
complexes in the instances \ref{enum:projective-bundle}
and~\ref{enum:blow-up} of Theorem~\ref{t:applications}. 
Without
using the terminology, Thomason implicitly establishes these
semi-orthogonal decompositions. 
He works with quasi-compact and quasi-separated schemes. 
It seems to us that his arguments should generalize to algebraic stacks,
which would give a proof of part \ref{enum:projective-bundle}
and~\ref{enum:blow-up} of Theorem~\ref{t:applications} different from
the one we give in this article.
\end{remark}

\subsection{Outline}
In Section~\ref{sec:preliminaries},
we summarize some basic facts about the derived category of an algebraic stack.
The relative notion of a Fourier--Mukai transform is defined in Section~\ref{sec:fourier-mukai}.
In two following sections, we develop the theory of conservative descent from the abstract point of view
of triangulated categories. In Section~\ref{sec:idemp-comon-mates}, we review the 2-categorical notions of mates and idempotent comonads.
We also show that a flat base change of a relative Fourier--Mukai
transform induces a compatibility between the associated
idempotent comonads as required later on for
  deducing Theorem~\ref{t:main-geometric-intro} from Theorem~\ref{t-main-categorical-intro}.
In Section~\ref{sec:semi-decomp-desc},
we explain that certain vanishing conditions on sequences of idempotent comonads are equivalent to the
 defining conditions of a semi-orthogonal decomposition,
which allows us to prove Theorem~\ref{t-main-categorical-intro}.
In Section~\ref{sec:applications}, we deduce the geometric version of the conservative descent theorem,
i.e., Theorem~\ref{t:main-geometric-intro}, and use this result to establish the semi-orthogonal decompositions
appearing in Theorem~\ref{t:applications}.

\subsection{Acknowledgments}
We thank Dustin Clausen for making us aware of Thomason's results
mentioned in Remark~\ref{rem:thomason}.
We also thank an anonymous referee for useful comments.
The first named author was partially supported by the Danish National
Research Foundation through the Niels Bohr Professorship of Lars Hesselholt,
by the Max Planck Institute for Mathematics in Bonn,
and by the DFG through SFB/TR 45.
The second named author was supported by the DFG through a postdoctoral fellowship and through SFB/TR 45.


\section{Preliminaries}
\label{sec:preliminaries}
We use the definition of algebraic space and algebraic stack given in the stacks
project~\cite[\sptag{025Y}, \sptag{026O}]{stacks-project}.
In particular, we do not assume that algebraic stacks be quasi-separated or have separated diagonals.
Algebraic stacks form a 2-category.
However, we follow the common practice to suppress 2-categorical details from the language and the notation.
For instance,
we usually simply write
\emph{commutative diagram}
or \emph{cartesian diagram}
instead of 2-commutative diagram or
2-cartesian diagram.

\subsection{Derived categories of algebraic stacks}
Given an algebraic stack $X$,
we consider its \emph{derived category}
$\D_\qc(X)$.
There are several approaches to constructing this category in the literature.
We briefly recall the one taken in \cite{lmb2000}.
The $\mathcal{O}_X$-modules in the \emph{lisse-étale}
topos $X_\liset$ form a Grothendieck abelian category.
We denote its (unbounded) derived category by $\D(X_\liset, \mathcal{O}_X)$.
Now $\D_\qc(X)$ is defined as the full subcategory of $\D(X_\liset, \mathcal{O}_X)$
whose objects are complexes with quasi-coherent cohomology
sheaves.

Recall that the category $\D_\qc(X)$ has the structure of a closed symmetric monoidal category,
whose operations we denote by
\begin{equation}
\label{eq-monoidal-ops}
-\otimes-,
\qquad
\sheafHom(-, -).
\end{equation}
Given an arbitrary morphism $f\colon X \to Y$ of algebraic stacks,
we get an induced adjoint pair of functors
\begin{equation}
\label{eq-derived-pushes}
f^*\colon\D_\qc(Y) \to \D_\qc(X), \qquad  f_*\colon\D_\qc(X) \to \D_\qc(Y).
\end{equation}
The precise construction of these functors is somewhat technical.
We refer to \cite[Section~1]{hr2017} for a detailed discussion.

\subsection{Concentrated morphisms}
It should be emphasized that even though the functors \eqref{eq-derived-pushes} exist in the generality
stated above,
they are not necessarily well behaved without further assumptions.
The situation becomes better if $f$ is assumed to be \emph{concentrated},
as defined by Hall--Rydh {\cite[Definition~2.4]{hr2017}}.
We recall the definition here.

\begin{definition}
\label{def-concentrated}
An algebraic stack $X$ is \emph{concentrated} if it is quasi-compact, quasi-separated
and has finite cohomological dimension.
Recall that an algebraic stack has finite cohomological dimension provided that there exists an integer $d$
such that for any quasi-coherent sheaf $\mathcal{F}$ of $\mathcal{O}_X$-modules,
we have $\HH^i(X, \mathcal{F}) = 0$ for every $i > d$.
A morphism $f\colon X \to Y$ is \emph{concentrated} if for any cartesian diagram
as in \eqref{eq-base-change} with $Y'$ affine, the stack $X'$ is concentrated.
\end{definition}

\begin{remark}
\label{rem-conc-base-change}
The property of a morphism of being concentrated is preserved under arbitrary base change and can be verified after a faithfully flat base change.
Moreover, an algebraic stack $X$ is concentrated if and only if it is concentrated over $\Spec \mathbb{Z}$.
This follows from \cite[Lemma~2.5]{hr2017}.
\end{remark}

\begin{remark}
\label{rem-conc-stacks}
Let $X$ be a quasi-compact and quasi-separated algebraic stack.
Assume furthermore that $X$ has finitely presented inertia and affine stabilizers.
Then $X$ is concentrated if and only if all stabilizers at points of positive characteristic are linearly
reductive.
This follows from a slightly more general result by Hall and Rydh~\cite[Theorem~C]{hr2015}.
\end{remark}

\begin{remark}
\label{rem-conc-representable}
A scheme or an algebraic space is concentrated if and only if it is quasi-compact and quasi-separated.
Indeed this is a special case of Remark~\ref{rem-conc-stacks}.
As a consequence, a representable morphism of algebraic stacks is
concentrated if and only if it is quasi-compact and quasi-separated. 
\end{remark}

Now let $f\colon X \to Y$ be a concentrated morphism of
algebraic stacks. 
Then the functor $f_*$ is the restriction of the corresponding derived functor
$\D(X_\liset, \mathcal{O}_X) \to\D(Y_\liset, \mathcal{O}_Y)$ (see \cite[Theorem~2.6]{hr2017}).
In this situation the functor $f_*$ has a right adjoint (\cite[Theorem~4.14]{hr2017}),
which we denote by
\begin{equation}
\label{eq-duality-functor}
f^\times \colon\D_\qc(Y) \to \D_\qc(X).
\end{equation}

\subsection{Perfect and pseudo-coherent complexes}
The theory of \emph{pseudo-coherent} and~\emph{perfect} complexes on a ringed topos
is worked out in \cite[Exposé~II]{sga6}
(see also \cite[\sptag{08G5}, \sptag{08FT}]{stacks-project} for the definitions).
In particular, these definitions apply to $(X_\liset, \mathcal{O}_X)$ when $X$ is an algebraic stack.
This gives us the following inclusions of full triangulated subcategories
\begin{equation}
\label{eq-subcats}
\D_\pf(X) \subset \D^\locbd_\pc(X) \subset \D_\pc(X) \subset \D_\qc(X),
\end{equation}
where $\D_\pf(X)$ is the subcategory of perfect complexes,
$\D_\pc(X)$ is the subcategory of pseudo-coherent complexes
and $\D^\locbd_\pc(X)$ is the subcategory of $\D_\pc(X)$
of complexes which locally have bounded cohomology.
Furthermore, the \emph{singularity category} for $X$ is defined as the Verdier quotient
\begin{equation}
\label{eq-sing}
\D_\sg(X) := \D^\locbd_\pc(X)/\D_\pf(X).
\end{equation}
\begin{remark}
The definition of perfect and pseudo-coherent complexes on a
general ringed topos is somewhat involved,
and it is sometimes convenient to instead use the following
characterization.

Let $X$ be an algebraic stack and let $\mathcal{F}$ be an
object in $\D(X_\liset, \mathcal{O}_X)$.
Then $\mathcal{F}$ is perfect (resp.~pseudo-coherent) if
and only if for any point $x \in X$ there exists a smooth neighborhood $U  \to X$ of $x$,
where $U$ is an affine scheme,
such that there exists a quasi-isomorphism $\mathcal{P} \to \mathcal{F}|_U$
where $\mathcal{P}$ is a bounded (resp.~bounded above) complex of finite locally free $\mathcal{O}_U$-modules.

Here the reduction to the local situation is obvious.
Assume that $U = \Spec R$.
The perfect (resp. pseudo-coherent) objects in $\D(\Mod(R))$ are
characterized as the complexes $M$ admitting a quasi-isomorphism
$P \to M$, where $P$ is a bounded (resp. bounded above) complex
of finitely generated projective $R$-modules (see \cite[\sptag{064U}]{stacks-project}).
Furthermore, we have an equivalence between the category of quasi-coherent modules on the ringed topos
$(U_\liset, \mathcal{O}_U)$ and the category $\Mod(R)$ of $R$-modules given by taking global sections.
This induces an equivalence
$\Rd\Gamma \colon \D_\qc(U) \to \D(\Mod(R))$ of derived categories.
We leave it as an exercise to the reader to verify that this equivalence preserves perfect and pseudo-coherent
complexes
(cf.\ \cite[\sptag{08EB}, \sptag{08E7}, \sptag{08HE}, \sptag{08HG}]{stacks-project} for the corresponding statement for the small étale topos).
\end{remark}

\begin{remark}
The category $\D^\locbd_\pc(X)$ is mostly interesting in the case when $X$ is noetherian,
in which case it coincides with the category $\D^\bd_\co(X)$ of
complexes with bounded, coherent cohomology.
\end{remark}

\begin{remark}
Note that by default, we do not use any derived decorations for the functors
\eqref{eq-monoidal-ops}, \eqref{eq-derived-pushes} and~\eqref{eq-duality-functor}.
Also, we use the same notation for the induced functors on the categories
\eqref{eq-subcats} and \eqref{eq-sing} provided that they exist.
The precise meaning of the symbols $\otimes, \sheafHom, f^*, f_*, f^\times$
should always be clear from the context.
\end{remark}


\section{Relative Fourier--Mukai transforms}
\label{sec:fourier-mukai}
In this section, we discuss Fourier--Mukai transforms in a relative setting.
The definition is a straightforward generalization of the usual concept,
but involves some technical conditions in order to ensure that such transforms
behave well under base change and that the induced functors preserve
perfect and locally bounded pseudo-coherent complexes.
The required conditions are fairly well understood if we restrict the discussion to 
quasi-compact and quasi-separated schemes,
but the situation for algebraic stacks seems to be more complicated.
We do not investigate these conditions systematically in this article.
Instead our strategy is to postulate the properties we require in Definition~\ref{def-properties}
and to prove that the types of morphism appearing in our applications satisfy these properties in
Proposition~\ref{prop-preserving}.

Consider a cartesian square
\begin{equation}
\label{eq-base-change}
\xymatrix{
X' \ar[r]^-v \ar[d]_-{g} & X \ar[d]^-f\\
Y' \ar[r]^-u & Y
}
\end{equation}
of algebraic stacks,
where, for the moment, we do not put any further conditions on the morphisms $f, g, u$ and~$v$.
The obvious isomorphism $g^*u^*\cong v^*f^*$
together with the adjunction counit $f^*f_* \to \id$ give a natural transformation
\begin{equation}
g^*u^*f_* \xrightarrow{\sim} v^*f^*f_* \to v^*. 
\end{equation}
By adjunction, we get an induced natural transformation
\begin{equation}
\label{eq-base-change-transformation}
u^*f_* \xrightarrow{} g_*v^*
\end{equation}
between functors $\D_\qc(X) \ra \D_\qc(Y')$
which we call the 
\emph{base change transformation for $f_*$ (along $u$)}. 
We are interested in situations when the base change transformation is an isomorphism.

\begin{proposition}[{\cite[Corollary~4.13]{hr2017}}]
\label{prop-concentrated-base-change}
Let $f\colon X \to Y$ be a concentrated morphism of algebraic
stacks appearing in a cartesian square \eqref{eq-base-change}.
Then the base change morphism \eqref{eq-base-change-transformation} for $f_*$ is an isomorphism
provided that $f$ and $u$ are tor-independent.
In particular, this holds if $f$ or $u$ is flat.
\end{proposition}

Now let $f\colon X \to Y$ be a concentrated morphism of algebraic
stacks appearing in a cartesian diagram \eqref{eq-base-change},
and assume that $u$ is flat. Then
we have a natural transformation
\begin{equation}
g_*v^*f^\times \xrightarrow{\sim} u^*f_*f^\times \to u^*
\end{equation}
constructed from the inverse of the base change transformation
\eqref{eq-base-change-transformation},
which is an isomorphism by Proposition~\ref{prop-concentrated-base-change},
and  the adjunction counit $f_*f^\times \to \id$.
By adjunction, we get an induced natural transformation
\begin{equation}
\label{eq-duality-base-change}
v^*f^\times \to g^\times u^*,
\end{equation}
which we call the \emph{base change transformation for $f^\times$
(along $u$)}.
Moreover, we have a natural transformation
\begin{equation}
f_*\left(f^\times(\mathcal{O}_Y) \otimes f^*(-)\right) \xrightarrow{\sim} f_*f^\times(\mathcal{O}_Y)\otimes (-) \to \id, 
\end{equation}
where the first morphism is induced by the projection formula, which holds since $f$ is concentrated by \cite[Corollary~4.12]{hr2017},
and the second by the adjunction counit.
By adjunction, we get an induced natural transformation
\begin{equation}
\label{eq-tensor}
f^\times(\mathcal{O}_Y) \otimes f^*(-) \to f^\times.
\end{equation}

Recall that a ring homomorphism $A \to B$ is \emph{perfect} \cite[\sptag{067G}]{stacks-project}
if $B$ admits a presentation of the form
$B \cong A[x_1, \cdots, x_n] / I$,
such that $B$ is perfect as an  $A[x_1, \cdots, x_n]$-module.
This notion extends, in the usual way, to a property of morphisms of algebraic stacks
by demanding that the property be \emph{fppf} local on both the source and the target
(see \cite[\sptag{0685}]{stacks-project}
for the corresponding situation for morphisms of schemes).
\begin{definition}
\label{def-properties}
Let $f\colon X \to Y$ be a proper, perfect and concentrated morphism of algebraic stacks.
We say that $f$ has property
\begin{enumerate}[label=\bf{P\arabic*},ref=P\arabic*]
\item
\label{enum:pty-quasi-perfect}
if $f_*$ preserves perfect complexes;
\item
\label{enum:pty-quasi-proper}
if $f_*$ preserves pseudo-coherent complexes;
\item
\label{enum:pty-base-change}
if \eqref{eq-tensor} is an isomorphism and the
base change transformation \eqref{eq-duality-base-change} for $f^\times$
along any flat morphism $u$ is an isomorphism.
\end{enumerate}
Moreover, we say that $f$ has any of the properties \ref{enum:pty-quasi-perfect}--\ref{enum:pty-base-change}
\emph{uniformly} provided that for any cartesian square \eqref{eq-base-change}
with $u$ flat, the morphism $g$ has the corresponding property.
\end{definition}

We are now ready to give the definition of a relative
Fourier--Mukai transform and to discuss its basic properties.

\begin{definition}
\label{def-fourier-mukai}
Let $X$ and $Y$ be algebraic stacks over some base algebraic
stack~$S$.
A \define{Fourier--Mukai transform} $\Phi\colon X \to
Y$ \define{over $S$}
is a diagram
\begin{equation}
\label{eq-fm-diagram}
\xymatrix{
& K \ar[dl]_p\ar[dr]^q& \\
X & & Y \\
}
\end{equation}
of algebraic stacks over $S$ together with an object $\mathcal{K}$ in $\D_\pf(K)$ such that
\begin{enumerate}
\item
  \label{enum:ppc}
the morphisms $p$ and $q$ are proper, perfect and concentrated
(see Definition~\ref{def-concentrated});
\item
  \label{enum:times-perf}
the object $q^\times\mathcal{O}_Y$ is perfect;
\item
  \label{enum:p-pres-perf-pcoh}
the morphism $p$ has properties \ref{enum:pty-quasi-perfect} and~\ref{enum:pty-quasi-proper} uniformly in the sense of Definition~\ref{def-properties};
\item
  \label{enum:q-pres-perf-pcoh-fbcx}
the morphisms $q$ has properties \ref{enum:pty-quasi-perfect}--\ref{enum:pty-base-change} uniformly in the sense of Definition~\ref{def-properties}.
\end{enumerate}
\emph{Note:
If we restrict the discussion to quasi-compact and
quasi-separated schemes then 
\ref{enum:p-pres-perf-pcoh}
 and \ref{enum:q-pres-perf-pcoh-fbcx}
 are implied by the other
conditions by Remark~\ref{rem-schemes} below.
We do not know if this is true in general.}
\end{definition}

Given a Fourier--Mukai transform $\Phi = (K, p, q, \mathcal{K}) \colon X \to Y $, we get an induced functor
\begin{equation}
\label{eq-fm-functor}
\D_\qc(X) \to \D_\qc(Y), \qquad \mathcal{F} \mapsto q_*(\mathcal{K}\otimes p^*(\mathcal{F})).
\end{equation}
This functor admits a right adjoint, which is given by
\begin{equation}
\label{eq-fm-functor-right}
\D_\qc(Y) \to \D_\qc(X), \qquad \mathcal{G} \mapsto p_*(\mathcal{K}^\vee\otimes q^\times(\mathcal{G})),
\end{equation}
where $\mathcal{K}^\vee := \sheafHom(\mathcal{K}, \mathcal{O}_K)$ denotes the dual of $\mathcal{K}$.

\begin{remark}
\label{rem:fm-functor}
By abuse of language, we also say that a functor $\Phi\colon \D_\qc(X) \to \D_\qc(Y)$
isomorphic to a functor of the form \eqref{eq-fm-functor} is a Fourier--Mukai transform.
When doing so, we always fix a corresponding geometric datum
$\Phi = (K, p, q, \mathcal{K})$ which is obvious from the
definition of the functor.
Note that we use the same symbol $\Phi$ for the functor and the geometric datum.
\end{remark}

The prominent feature of relative Fourier--Mukai transforms
is that they admit a natural notion of flat base change.
More precisely, if $\Phi = (K, p, q, \mathcal{K})$ is a Fourier--Mukai transform over $S$ and $S' \to S$ is a flat morphism,
then we form the base change $\Phi_{S'} = (K', p', q', \mathcal{K}')$, where
\begin{equation}
\xymatrix{
& K' \ar[dl]_-{p'}\ar[dr]^-{q'}& \\
X' & & Y' \\
}
\end{equation}
is the base change of the diagram \eqref{eq-fm-diagram} along $S' \to S$,
and $\mathcal{K}'$ is the pull-back of $\mathcal{K}$ along the induced morphism $K' \to K$.
The conditions on the morphisms $p$ and $q$ assert that $\Phi'$
is again a Fourier--Mukai transform.

\begin{proposition}
\label{prop-fm-restriction}
Let $X, Y$ be algebraic stacks over an algebraic stack $S$,
and let $\Phi = (K, p, q, \mathcal{K})\colon X \to Y$ be a Fourier--Mukai transform over $S$.
Then the functors \eqref{eq-fm-functor} and~\eqref{eq-fm-functor-right}
preserve perfect and locally bounded pseudo-coherent complexes.
\end{proposition}
\begin{proof}
Derived pull-backs always preserve perfect and pseudo-coherent complexes.
Since the morphisms $p$ and $q$ are assumed to be perfect, they have finite tor-dimension.
Hence $p^*$ and $q^*$ preserve locally bounded complexes.
By assumption, both $p$ and $q$ satisfy \ref{enum:pty-quasi-perfect} and~\ref{enum:pty-quasi-proper}.
Hence all the functors $p^*, q^*, p_*, q_*$ preserve perfect complexes and
locally bounded pseudo-coherent complexes.
Since also $q$ satisfies \ref{enum:pty-base-change}, we have $q^\times \cong q^\times(\mathcal{O}_Y) \otimes q^*(-)$.
By assumption $q^\times(\mathcal{O}_Y)$ is perfect, so also $q^\times$ preserves perfect complexes
and locally bounded pseudo-coherent complexes.
Now the proposition follows from the fact that the functors \eqref{eq-fm-functor} and~\eqref{eq-fm-functor-right}
are compositions of the functors $p^*, q_*, p_*, q^\times$ and
tensoring with the perfect complexes~$\mathcal{K}$ and~$\mathcal{K}^\vee$.
\end{proof}

Note that the restriction of the functor \eqref{eq-fm-functor} to the categories of perfect complexes also admits a left adjoint.
Explicitly, this is given by 
\begin{equation}
\label{eq-fm-pf-left}
\D_\pf(Y) \to \D_\pf(X), \qquad \mathcal{G} \mapsto p_\times(\mathcal{K}^\vee\otimes q^*(\mathcal{G})),
\end{equation}
where $p_\times \colon \mathcal{G} \mapsto p_*(\mathcal{G}^\vee)^\vee$ denotes the left adjoint of
$p^*\colon \D_\pf(K) \to \D_\pf(X)$ (see e.g.~\cite[Lemma~4.3]{bls2016}).

We conclude the section by discussing situations when
the conditions
\ref{enum:pty-quasi-perfect}--\ref{enum:pty-base-change}
are satisfied.

\begin{remark}
\label{rem-schemes}
For a moment, let us restrict attention to the category of quasi-compact and quasi-separated schemes.
Then the properties in Definition~\ref{def-properties} are well understood.
They are treated thoroughly by Lipman and Neeman in \cite{ln2007}.
A good overview is also given by Lipman in \cite[Section~4.7]{lipman2009}.
Here we give a brief summary of the results that are relevant to us.

Assume that $f\colon X \to Y$ is a morphism of quasi-compact and quasi-separated schemes.
In particular, such a morphism is concentrated by Remark~\ref{rem-conc-representable}.
If $f$ is proper and perfect, then it automatically satisfies all properties in Definition~\ref{def-properties}
uniformly.
Indeed, since it is proper and \emph{pseudo-coherent} it is also \emph{quasi-proper}, i.e., satisfies
\ref{enum:pty-quasi-proper} by \cite[Corollary~4.3.3.2]{lipman2009}.
Since it has finite tor-dimension, it is also
\emph{quasi-perfect},
by \cite[Theorem~1.2]{ln2007}, i.e., satisfies \ref{enum:pty-quasi-perfect},
by \cite[Proposition~2.1]{ln2007}, which also tells us that \eqref{eq-tensor} is an isomorphism.
Hence it also satisfies \ref{enum:pty-base-change} 
by~\cite[Theorem~4.7.4]{lipman2009}.
In particular, conditions \ref{enum:p-pres-perf-pcoh} and
\ref{enum:q-pres-perf-pcoh-fbcx} in
Definition~\ref{def-fourier-mukai} are redundant in this
situation.
\end{remark}

\begin{remark}
\label{rem-stacks}
The situation is less explored if we consider arbitrary concentrated morphisms $f\colon X \to Y$ of
algebraic stacks.
If $f$ is proper, perfect and representable by schemes,
then it satisfies \ref{enum:pty-quasi-perfect} and~\ref{enum:pty-quasi-proper} uniformly.
Indeed, these properties can be verified locally on the target, so this follows from Remark~\ref{rem-schemes}.
If, in addition, $f$ is finite, then also
\ref{enum:pty-base-change} holds uniformly
by~\cite[Theorem~4.14(4)]{hr2017}.
The properties in Definition~\ref{def-properties} in the context of algebraic stacks
are further explored by Neeman in the recent preprint~\cite{neeman2017}.
\end{remark}

\begin{remark}
Let us restrict the discussion to the category of smooth and projective schemes over an
algebraically closed field.
Given objects $X$, $Y$ in this category and an object $\mathcal{K}$ in $\D_\pf(X \times Y)$,
we obtain the Fourier--Mukai transform $\Phi = (X\times Y, p, q, \mathcal{K}) \colon X \to Y$,
where $p$ and $q$ are the canonical projections.
Indeed, by Remark~\ref{rem-schemes},
it is enough to verify items \ref{enum:ppc} and~\ref{enum:times-perf} of Definition~\ref{def-fourier-mukai}.
Here \ref{enum:ppc} is clear (cf.~proof of Proposition~\ref{prop-preserving} below).
Furthermore, \ref{enum:times-perf} follows from
Grothendieck--Verdier duality since $q^\times\mathcal{O}_Y \cong
q^!\mathcal{O}_Y \cong \Sigma^{\dim X}p^*\omega_X$ (see e.g.~\cite[Section~3.4]{huybrechts2006}). 
Usually the term Fourier--Mukai transform,
as defined for instance in \cite[Definition~5.1]{huybrechts2006},
refers to the induced functor $\Phi\colon \D_\pf(X) \to \D_\pf(Y)$ obtained from such a datum.
In particular, our notion is a direct generalization of the standard notion.
\end{remark}

The next proposition summarizes what we need for the applications considered in this article.

\begin{proposition}
\label{prop-preserving}
Let $f\colon X \to Y$ be a morphism of algebraic stacks.
Then $f$ is perfect, proper, concentrated and satisfies properties
\ref{enum:pty-quasi-perfect} and \ref{enum:pty-quasi-proper} uniformly in the following cases:
\begin{enumerate}
\item
  \label{enum:proj}
$f$ is a projectivized vector bundle;
\item
  \label{enum:blowup}
$f$ is a blow-up in a regular sheaf of ideals;
\item
  \label{enum:regimm}
$f$ is a regular closed immersion;
\item
  \label{enum:gerbe}
$f$ is $\mu_n$-gerbe;
\item
  \label{enum:root}
$f$ is a root stack in an effective Cartier divisor.
\end{enumerate}
Moreover, if $f$ is a regular closed immersion then $f^\times(\mathcal{O}_Y)$ is perfect and $f$ satisfies
\ref{enum:pty-base-change} uniformly.
\end{proposition}
\begin{remark}
\label{rem-regular-immersion}
We follow the definition of \emph{regular immersion} given in SGA6 \cite[Exposé~VII, Définition~1.4]{sga6}.
This is what is called a \define{Koszul-regular immersion} in the stacks project \cite[\sptag{0638}]{stacks-project}.
A \emph{regular sheaf of ideals} is simply an ideal sheaf corresponding to a regular closed immersion.
The property of being a regular immersion is local on the target for the \emph{fpqc}
topology~\cite[Exposé~VII, Proposition~1.5]{sga6}.
In particular, the definition automatically extends to algebraic stacks.
\end{remark}
\begin{proof}[Proof of Proposition~\ref{prop-preserving}]
The morphism $f$ is clearly proper and of finite presentation in all the cases \ref{enum:proj}--\ref{enum:root}.
It is also perfect by flatness and by being of finite
presentation in the cases \ref{enum:proj}, \ref{enum:gerbe} and~\ref{enum:root} and
by \cite[Exposé VII, Proposition~1.9]{sga6} in the cases \ref{enum:blowup} and~\ref{enum:regimm}.
Moreover, $f$ is concentrated by Remark~\ref{rem-conc-base-change}, \ref{rem-conc-stacks} and~\ref{rem-conc-representable}.

In the cases \ref{enum:proj}--\ref{enum:regimm},
$f$ is representable by schemes,
so \ref{enum:pty-quasi-perfect} and \ref{enum:pty-quasi-proper} hold uniformly by Remark~\ref{rem-stacks}.
Assume that we are in one of the cases \ref{enum:gerbe} or~\ref{enum:root}.
The push-forward $f_*$ preserves perfects by \cite[Lemma~4.5]{bls2016},
so $f$ has property~\ref{enum:pty-quasi-perfect}.
Furthermore, from the proof of the same lemma,
we may, after an appropriate base change, assume that 
the category of quasi-coherent $\mathcal{O}_X$-modules is equivalent
to the category of $R$-modules for some not necessarily commutative
ring $R$.
Hence any pseudo-coherent complex of $\mathcal{O}_X$-modules is quasi-isomorphic
to a bounded above complex of finite locally free $\mathcal{O}_X$-modules.
Indeed, the proof for this is identical to the proof in the case when $R$ is commutative
(cf.~\cite[\sptag{068R}]{stacks-project}).
Since $f_*$ is bounded and preserves perfect complexes,
this implies that $f$ has property~\ref{enum:pty-quasi-proper}.
Hence $f$ satisfies property \ref{enum:pty-quasi-perfect} and~\ref{enum:pty-quasi-proper} uniformly,
since the property of being a $\mu_n$-gerbe or a root stack in a Cartier divisor
is stable under flat base change.

Finally, assume that $f$ is a regular closed immersion.
Then $f$ is finite so \ref{enum:pty-base-change} is satisfied
uniformly by Remark~\ref{rem-stacks}
and $f^\times(\mathcal{O}_Y)$ is perfect by~\cite[Lemma~4.1]{bls2016}.
\end{proof}


\section{Mates and idempotent comonads}
\label{sec:idemp-comon-mates}

The base change transformations \eqref{eq-base-change-transformation}
and \eqref{eq-duality-base-change} are instances of what is
called a \emph{mate} in category theory. We recall some aspects of the
calculus of mates in \ref{sec:mates} and use
this theory
in \ref{sec:four-mukai-transf}
to show
that Fourier--Mukai transforms
and their right adjoints respect flat base change in a certain sense.
In \ref{sec:cons-desc-fully} we prove \emph{conservative descent for
fully faithfulness} and deduce that fully faithfulness of
Fourier--Mukai transforms can be tested after a faithfully flat
base change. 

The essential image of a fully faithful Fourier--Mukai transform
is a \emph{right admissible subcategory},
and hence the induced projection functor onto this subcategory is an example of what is
called an
\emph{idempotent comonad}. If a Fourier--Mukai transform and its flat
base change are both fully faithful, their induced idempotent
comonads are \emph{compatible}. This result is a consequence of 
a general categorical result established in 
\ref{sec:comp-from-mates} after the necessary prerequisites on 
idempotent comonads (and monads) and their compatibilities are
explained in 
\ref{sec:comon-idemp-comon},
\ref{sec:comp-idemp-comon} and
\ref{sec:monads}.

The general results in this section are formulated for an
arbitrary fixed 2-category.  The words ``object'',
``1-morphism'', ``2-morphism'' refer to this 2-category. We use
the symbols $\ra$ and $\Ra$ for 1-morphisms and 2-morphisms, 
respectively. The geometric examples take place in the 2-category
of triangulated categories.

\subsection{Mates with geometric examples}
\label{sec:mates}

For the convenience of the reader, we recall some results on
mates from
\cite[\S 2]{kelly-street-review-2-categories}.
We frequently consider the following situation.

\begin{setting}
  \label{set:two-adjunctions}
  Let
  \begin{equation}
    \label{eq:two-adjunctions}
    \begin{tikzpicture}[baseline=(current bounding box.center),
      description/.style={fill=white,inner sep=2pt},
      natural/.style={double,double equal sign distance,-implies} ]
      
      \matrix (m) [matrix of math nodes, row sep=4em,
      column sep=4em, text height=1.5ex, text depth=0.25ex]
      { \mathcal{C}' & \mathcal{C}\\
        \mathcal{D}' & \mathcal{D}\\};
      \path[->,font=\scriptsize]
      (m-1-2) edge node[above] {$F$} (m-1-1)
      (m-2-2) edge node[below] {$G$} (m-2-1)
      (m-2-1.120) edge node[left] {$L'$} (m-1-1.-120)
      (m-1-1.-80) edge node[right] {$R'$} (m-2-1.80)
      (m-2-2.120) edge node[left] {$L$} (m-1-2.-120)
      (m-1-2.-80) edge node[right] {$R$} (m-2-2.80);
    \end{tikzpicture}
  \end{equation}
  be a diagram of objects and 1-morphisms and let
  $(L', R', \eta', \epsilon')$ 
  and
  $(L, R, \eta, \epsilon)$ be adjunctions where $\eta \colon
  \id_\mathcal{D} 
  \Ra RL$ and
  $\eta'$ are the unit 2-morphisms 
  and $\epsilon \colon LR \Ra
  \id_\mathcal{C}$ and $\epsilon'$ are the counit 2-morphisms.
\end{setting}

\begin{lemma}
  [{\cite[Proposition~2.1]{kelly-street-review-2-categories}}]
  \label{l:mates}
  In Setting~\ref{set:two-adjunctions}
  there is canonical bijection
  \begin{equation*}
    \{\text{2-morphisms $\alpha \colon L'G \Ra FL$}\} 
    \sira
    \{\text{2-morphisms $\beta \colon GR \Ra R'F$}\}
  \end{equation*}
  sending a 2-morphism $\alpha \colon L'G \Ra FL$ to 
  the 2-morphism $\ol{\alpha}$ defined in \eqref{eq:mate} below.
\end{lemma}

\begin{definition}
  \label{d:mate}
  We call $\ol{\alpha}$ the \define{mate} of $\alpha$, in the
  notation of Lemma~\ref{l:mates}.
\end{definition}

The 2-morphisms in this lemma are illustrated by the following
diagrams.
\begin{equation}
  \begin{tikzpicture}[baseline=(current bounding box.center),
    description/.style={fill=white,inner sep=2pt},
    natural/.style={double,double equal sign distance,-implies} ]
    
    \matrix (m) [matrix of math nodes, row sep=4em,
    column sep=4em, text height=1.5ex, text depth=0.25ex]
    { \mathcal{C}' & \mathcal{C}\\
      \mathcal{D}' & \mathcal{D}\\};
    \path[->,font=\scriptsize]
    (m-1-2) edge node[above] {$F$} (m-1-1)
    (m-2-2) edge node[below] {$G$} (m-2-1)
    (m-2-1) edge node[left] {$L'$} (m-1-1)
    (m-2-2) edge node[right] {$L$} (m-1-2)
    
    (m-2-1) edge[natural,shorten >=0.5em,shorten <=0.5em]
      node[pos=0.5, above left] {$\alpha$} (m-1-2); 
  \end{tikzpicture}
  \hspace{2em}
  \begin{tikzpicture}[baseline=(current bounding box.center),
    description/.style={fill=white,inner sep=2pt},
    natural/.style={double,double equal sign distance,-implies} ]
    
    \matrix (m) [matrix of math nodes, row sep=4em,
    column sep=4em, text height=1.5ex, text depth=0.25ex]
    { \mathcal{C}' & \mathcal{C}\\
      \mathcal{D}' & \mathcal{D}\\};
    \path[->,font=\scriptsize]
    (m-1-2) edge node[above] {$F$} (m-1-1)
    (m-2-2) edge node[below] {$G$} (m-2-1)
    (m-1-1) edge node[left] {$R'$} (m-2-1)
    (m-1-2) edge node[right] {$R$} (m-2-2)
    
    (m-2-2) edge[natural,shorten >=0.5em,shorten <=0.5em]
    node[pos=0.5, above right]{$\beta$} (m-1-1); 
  \end{tikzpicture}
\end{equation}

\begin{proof}
  Let $\alpha \colon L'G \Ra FL$ be given.
  The left diagram above can be expanded to the following diagram.
  \begin{equation}
    \begin{tikzpicture}[baseline=(current bounding box.center),
      description/.style={fill=white,inner sep=2pt},
      natural/.style={double,double equal sign distance,-implies} ]
      
      \matrix (m) [matrix of math nodes, row sep=4em,
      column sep=2em, text height=1.5ex, text depth=0.25ex]
      {          & \mathcal{C}' && \mathcal{C} && \mathcal{C}\\
        \mathcal{D}' && \mathcal{D}' && \mathcal{D}\\};
      \path[->,font=\scriptsize]
      (m-1-4) edge node[above]{$F$} node[pos=0.18, below](end){} (m-1-2)
      (m-2-5) edge node[below]{$G$} node[pos=0.82, above](start){} (m-2-3)
      (m-2-3) edge node[right] {$L'$} (m-1-2)
      (m-2-5) edge node[left] {$L$} (m-1-4)
      (m-1-2) edge node[left] {$R'$} (m-2-1)
      (m-1-6) edge node[right] {$R$} (m-2-5)

      (m-2-3) edge node[below]{$\id$} node[above](gleichD'){} (m-2-1)
      (m-1-6) edge node[above]{$\id$} node[below](gleichC){} (m-1-4);
            
      \path[->,font=\scriptsize]
      (start) edge[natural]
      node[right] {$\alpha$} (end)
      (gleichD') edge[natural,shorten >=0.2em]
      node[pos=0.4, right] {$\eta'$} (m-1-2) 
      (m-2-5) edge[natural,shorten <=0.2em]
      node[pos=0.6, right] {$\epsilon$} (gleichC); 
    \end{tikzpicture}
  \end{equation}
  Define the 2-morphism $\ol{\alpha}$ as the composition
  \begin{equation}
    \label{eq:mate}
    \ol{\alpha} := (R'F\epsilon) \circ (R'\alpha R) \circ
    (\eta'GR) \colon GR \xRa{\eta' GR} R'L'GR \xRa{R'\alpha R}
    R'FLR \xRa{R'F 
      \epsilon} R'F.
  \end{equation}
  Similarly, we associate to each 2-morphism $\beta \colon
  GR \Ra R'F$ the following composition of 2-morphisms 
  \begin{equation}
    \label{eq:inverse-mate}
    L'G \xRa{L'G\eta} L'GRL \xRa{L'\beta L} 
    L'R'FL \xRa{\epsilon' FL} FL.
  \end{equation}
  The triangle identities satisfied by the adjunction data show
  that this map is inverse to the map $\alpha \mapsto
  \ol{\alpha}$.
\end{proof}

\begin{example}
  \label{eg:base-changes-*-times}
  Given a cartesian diagram
  \eqref{eq-base-change} of algebraic stacks
  consider the obvious 2-isomorphism $\alpha \colon g^*u^* \siRa
  v^*f^*$. Its mate $\ol{\alpha} \colon u^*f_* \Ra g_*v^*$
  (with respect to the adjunctions $(f^*, f_*)$ and $(g^*, g_*)$)
  is the 
  base change transformation 
  \eqref{eq-base-change-transformation}
  for~$f_*$.

  Assume that $f$ is concentrated.
  Then the mate $\ol{\alpha}$ has an inverse $\beta:= \ol{\alpha}^{-1}$,
  and we have adjunctions $(f_*, f^\times)$ and $(g_*, g^\times)$.
  We obtain the mate $\ol{\beta} \colon v^* f^\times
  \Ra g^\times u^*$, which is the 
  base change transformation 
  \eqref{eq-duality-base-change}
  for~$f^\times$.
\end{example}

\begin{example}
  \label{eg:tensor-perfect-cx}
  Let $u \colon X' \ra X$ be a morphism of algebraic stacks and 
  let $\mathcal{K} \in \D_\qc(X)$.
  Consider the obvious 2-isomorphism
  \begin{equation}
    \label{eq:lambda}
    \lambda \colon u^*\mathcal{K} \otimes u^*(-)
    \xRa{\sim}
    u^*(\mathcal{K} \otimes -)
  \end{equation}
  between triangulated functors $\D_\qc(X) \ra \D_\qc(X')$.
  Form its mate
  \begin{equation}
    \label{eq:lambda-mate}
    u^* \sheafHom(\mathcal{K}, -)
    \Ra \sheafHom(u^*\mathcal{K}, u^*(-))
    \colon \D_\qc(X) \ra \D_\qc(X')
  \end{equation}
  with respect to the adjunctions
  $(\mathcal{K} \otimes -, \sheafHom(\mathcal{K},-))$ and
  $(u^*\mathcal{K} \otimes -, \sheafHom(u^*\mathcal{K},-))$.

  Assume that $\mathcal{K}$ is perfect.
  Then it is straightforward to verify that the mate \eqref{eq:lambda-mate}
  is a 2-isomorphism.
  Indeed, this follows by a standard reduction to the affine case,
  using the fact that $\sheafHom(\mathcal{K},-)$ coincides with
  $\Rd\sheafHom_{\mathcal{O}_X}(\mathcal{K},-)$ on $\D(X_\liset, \mathcal{O}_X)$
  when $\mathcal{K}$ is perfect (see~\cite[Lemma~4.3]{hr2017}).
  Alternatively, we can use the adjunctions
  $(\mathcal{K} \otimes -, \mathcal{K}^\vee \otimes -)$ and
  $(u^*\mathcal{K} \otimes -, (u^*\mathcal{K})^\vee) \otimes -)$ 
  and obtain as the mate of $\lambda$ the 2-isomorphism
  \begin{equation}
    \label{eq:lambda-mate-dual}
    \ol{\lambda} \colon u^*(\mathcal{K}^\vee \otimes -)
    \siRa (u^*\mathcal{K})^\vee \otimes u^*(-).
  \end{equation}
\end{example}

The following technical lemma on mates is used in the proofs of
Proposition~\ref{p:conservative-descent-ff} and 
Lemma~\ref{l:idempotent-comonads-compatible}.
The reader may ignore it for now.

\begin{lemma}
  \label{l:counits-compatible-precursor}
  In Setting~\ref{set:two-adjunctions}
  let $\alpha \colon L'G \Ra FL$ be a
  2-morphism and let 
  $\ol{\alpha} \colon GR \Ra R'F$ be its mate.
  Then 
  the diagrams
  \begin{equation}
    \label{eq:alphaR-L'olalpha}
    \xymatrix{
      & {FLR} \ar@{=>}[dl]_-{F\epsilon} &&&
      {L'GR} \ar@{=>}[lll]_-{\alpha R}\\
      {F} &&
      {L'R'FLR}
      \ar@{=>}[ul]_-{\epsilon'FLR}
      \ar@{=>}[dl]_-{L'R'F\epsilon} &
      {L'R'L'GR} 
      \ar@{=>}[l]_-{L'R'\alpha R} 
      \ar@{=>}[ur]^-{\epsilon'L'GR} \\
      &
      {L'R'F} \ar@{=>}[ul]_-{\epsilon' F} &&&
      {L'GR} 
      \ar@{=>}^-{\id_{L'GR}}[uu]
      \ar@{=>}[lu]^-{L'\eta'GR} 
      \ar@{=>}[lll]_-{L'\ol{\alpha}}
    }
  \end{equation}
  and
  \begin{equation}
    \label{eq:unit-compatible-precursor}
    \xymatrix{
      & {GRL} 
      \ar@{=>}[rrrr]^-{\ol{\alpha}L} 
      \ar@{=>}[rd]^-{\eta'GRL} &&&& 
      {R'FL} 
      \ar@{=>}[ld]_-{\eta'R'FL}
      \ar@{=>}[dd]_-{\id_{R'FL}}\\
      {G} \ar@{=>}[ur]^-{G\eta} \ar@{=>}[dr]^-{\eta'G} && 
      {R'L'GRL} \ar@{=>}[rr]^-{R'L'\ol{\alpha}L} && 
      {R'L'R'FL} \ar@{=>}[rd]_-{R'\epsilon'FL}\\
      & {R'L'G} 
      \ar@{=>}[ru]^-{R'L'G\eta} 
      \ar@{=>}[rrrr]^-{R'\alpha}
      &&&& {R'FL}
    }
  \end{equation}
  are commutative.
\end{lemma}

\begin{proof}
  The upper trapezoids and the left rhombi are obviously
  commutative. The triangles on the right are commutative by the
  triangle identities of the adjunctions.
  The lower trapezoids are commutative by the definition of
  the bijection in Lemma~\ref{l:mates}.
\end{proof}

The following result explains how taking mates is compatible with
compositions.

\begin{lemma}
  [{\cite[Proposition~2.2]{kelly-street-review-2-categories}}]
  \label{l:stacking-squares}
  Let
  \begin{equation}
    \label{eq:two-adjunctions-stacked}
    \begin{tikzpicture}[baseline=(current bounding box.center),
      description/.style={fill=white,inner sep=2pt},
      natural/.style={double,double equal sign distance,-implies}]
 
      \matrix (m) [matrix of math nodes, row sep=4em,
      column sep=4em, text height=1.5ex, text depth=0.25ex]
      { \mathcal{C}' & \mathcal{C}\\
        \mathcal{D}' & \mathcal{D}\\
        \mathcal{E}' & \mathcal{E}\\};
      \path[->,font=\scriptsize]
      (m-1-2) edge node[above] {$F$} (m-1-1)
      (m-2-2) edge node[above] {$G$} (m-2-1)
      (m-3-2) edge node[above] {$H$} (m-3-1)

      (m-2-1.120) edge node[left] {$L'$} (m-1-1.-120)
      (m-1-1.-80) edge node[right] {$R'$} (m-2-1.80)
      (m-2-2.120) edge node[left] {$L$} (m-1-2.-120)
      (m-1-2.-80) edge node[right] {$R$} (m-2-2.80)

      (m-3-1.120) edge node[left] {$M'$} (m-2-1.-120)
      (m-2-1.-80) edge node[right] {$P'$} (m-3-1.80)
      (m-3-2.120) edge node[left] {$M$} (m-2-2.-120)
      (m-2-2.-80) edge node[right] {$P$} (m-3-2.80);
    \end{tikzpicture}
    \hspace{2em}
    \begin{tikzpicture}[baseline=(current bounding box.center),
      description/.style={fill=white,inner sep=2pt},
      natural/.style={double,double equal sign distance,-implies} ]
      
      \matrix (m) [matrix of math nodes, row sep=4em,
      column sep=4em, text height=1.5ex, text depth=0.25ex]
      { \mathcal{C}' & \mathcal{C}\\
        \mathcal{D}' & \mathcal{D}\\
        \mathcal{E}' & \mathcal{E}\\};
      \path[->,font=\scriptsize]
      (m-1-2) edge node[above] {$F$} (m-1-1)
      (m-2-2) edge node[above] {$G$} (m-2-1)
      (m-3-2) edge node[above] {$H$} (m-3-1)

      (m-2-1) edge node[left] {$L'$} (m-1-1)
      (m-2-2) edge node[right] {$L$} (m-1-2)

      (m-3-1) edge node[left] {$M'$} (m-2-1)
      (m-3-2) edge node[right] {$M$} (m-2-2)

      (m-2-1) edge[natural,shorten >=0.5em,shorten <=0.5em]
      node[pos=0.5, above left]{$\alpha$} (m-1-2)
      (m-3-1) edge[natural,shorten >=0.5em,shorten <=0.5em]
      node[pos=0.5, above left]{$\beta$} (m-2-2); 
    \end{tikzpicture}
  \end{equation}
  be diagrams in a 2-category and let
  $(L',R')$, $(L,R)$, $(M',P')$, $(M,P)$ 
  be adjunctions.
  Then the mate of $(\alpha M) \circ (L'\beta)$ is
  $(P'\ol{\alpha}) \circ (\ol{\beta}R)$.

  In particular, if $\alpha$, $\ol{\alpha}$, $\beta$,
  $\ol{\beta}$ are 2-isomorphisms, then
  $(\alpha M) \circ (L'\beta)$ and its mate are 2-isomorphisms.
\end{lemma}

\begin{proof}
  See \cite[Proposition~2.2]{kelly-street-review-2-categories}.
\end{proof}

\subsection{Fourier--Mukai transforms and flat base change}
\label{sec:four-mukai-transf}

\begin{setting}
  \label{set:FM-flat-BC}
  Let $X$ and $Y$ be algebraic stacks over a base algebraic stack $S$.
  Let $\Phi = (K, p, q, \mathcal{K}) \colon X \ra Y$ be a
  Fourier--Mukai transform over $S$.
  Let $u \colon S' \ra S$ be a flat morphism of algebraic stacks.
  The following diagram with cartesian squares is obtained from 
  $X \xla{p} K \xra{q} Y$ by base change along $u$.
  \begin{equation}
    \xymatrix{
      {X'} \ar[d]^-{f} 
      & {K'} \ar[d]^-{h} \ar[l]_-{p'} \ar[r]^-{q'}
      & {Y'} \ar[d]^-{g}
      \\
      {X}
      & {K} \ar[l]_-{p} \ar[r]^-{q}
      & {Y} 
    }
  \end{equation}
  Then $\Phi' = (K', p', q', \mathcal{K}':=h^*\mathcal{K})
  \colon X' \ra Y'$ is the base change of $\Phi$ along $u$.
  We denote the right adjoint functors of the Fourier--Mukai
  transforms $\Phi \colon \D_\qc(X) \ra \D_\qc(Y)$ and $\Phi'
  \colon \D_\qc(X') \ra \D_\qc(Y')$ by $\Psi$ and $\Psi'$,
  respectively (cf.~\eqref{eq-fm-functor-right}).
\end{setting}
  
In Setting~\ref{set:FM-flat-BC} we obtain the following diagram 
where the 2-isomorphisms $\alpha$ and $\lambda$ are the obvious
ones, and the 2-isomorphism $\beta$ is the inverse of the base change
transformation for $q_*$ along the flat morphism $g$.
Note that this base change transformation is indeed a 2-isomorphism by
Proposition~\ref{prop-concentrated-base-change} since $q$ is concentrated.
\begin{equation}
  \label{eq:stacked-2-isos}
  \begin{tikzpicture}[baseline=(current bounding box.center),
    description/.style={fill=white,inner sep=2pt},
    natural/.style={double,double equal sign distance,-implies} ]
    
    \matrix (m) [matrix of math nodes, row sep=4em,
    column sep=4em, text height=1.5ex, text depth=0.25ex]
    { \D_\qc(Y') & \D_\qc(Y)\\
      \D_\qc(K') & \D_\qc(K)\\
      \D_\qc(K') & \D_\qc(K)\\
      \D_\qc(X') & \D_\qc(X)\\};
    \path[->,font=\scriptsize]
    (m-1-2) edge node[above] {$g^*$} (m-1-1)
    (m-2-2) edge node[above] {$h^*$} (m-2-1)
    (m-3-2) edge node[above] {$h^*$} (m-3-1)
    (m-4-2) edge node[above] {$f^*$} (m-4-1)
    (m-2-1) edge node[left] {$q'_*$} (m-1-1)
    (m-2-2) edge node[right] {$q_*$} (m-1-2)
    (m-3-1) edge node[left] {$\mathcal{K}' \otimes -$} (m-2-1)
    (m-3-2) edge node[right] {$\mathcal{K} \otimes -$} (m-2-2)
    (m-4-1) edge node[left] {$p'^*$} (m-3-1)
    (m-4-2) edge node[right] {$p^*$} (m-3-2)
    
    (m-2-1) edge[natural,shorten >=0.5em,shorten <=0.5em]
    node[above left]{$\beta$} node[below right]{$\sim$}(m-1-2) 
    (m-3-1) edge[natural,shorten >=0.5em,shorten <=0.5em]
    node[above left]{$\lambda$} node[below right]{$\sim$}(m-2-2)
    (m-4-1) edge[natural,shorten >=0.5em,shorten <=0.5em]
    node[above left]{$\alpha$} node[below right]{$\sim$}(m-3-2); 
  \end{tikzpicture}
\end{equation}
The vertical compositions of the two columns in this diagram are
the Fourier--Mukai transforms $\Phi$ and $\Phi'$.

\begin{proposition}[Flat base change for Fourier--Mukai transforms]
  \label{p:flat-BC-FM}
  In Setting~\ref{set:FM-flat-BC} consider the 2-isomorphism
  \begin{equation}
    \label{eq:6}
    \Phi' f^* \siRa g^* \Phi.
  \end{equation}
  obtained from the three 2-isomorphisms in diagram
  \eqref{eq:stacked-2-isos}. Then its mate 
  \begin{equation}
    \label{eq:7}
    f^* \Psi \siRa \Psi' g^*
  \end{equation}
  is a 2-isomorphism as well.
\end{proposition}

\begin{proof}
  The mate of $\alpha$ is a 2-isomorphism since $f$ is flat and
  $p$ is concentrated
  (Proposition~\ref{prop-concentrated-base-change} and
  Example~\ref{eg:base-changes-*-times}). The mate of $\lambda$
  is a 2-isomorphism by Example~\ref{eg:tensor-perfect-cx} since
  $\mathcal{K}$ is perfect. The
  mate of $\beta$ is a 2-isomorphism since it is the base change
  transformation for $q^\times$ 
  (see Example~\ref{eg:base-changes-*-times}) which is a 2-isomorphism
  since $q$ satisfies property~\ref{enum:pty-base-change} from Definition~\ref{def-properties}.
  Hence Lemma~\ref{l:stacking-squares} yields the 
  statement.
\end{proof}

\subsection{Conservative descent for fully faithfulness}
\label{sec:cons-desc-fully}

\begin{definition}
  \label{d:conservative}
  A functor $G \colon \mathcal{D} \ra \mathcal{D}'$ between
  categories is called \define{conservative} if it reflects
  isomorphisms: if $d \colon D \ra D'$ is any morphism in
  $\mathcal{D}$ such that $G(d)$ is an isomorphism, then $d$ is
  an isomorphism.
\end{definition}

\begin{example}
  \label{eg:conservative}
  A triangulated functor is conservative if and only if it
  reflects zero objects.
  For example, if $g \colon Y' \ra Y$ is a faithfully flat
  morphism of algebraic stacks, then $g^* \colon \D_\qc(Y) \ra
  \D_\qc(Y')$ is conservative.
\end{example}

\begin{proposition}
  [Conservative descent for fully faithfulness]
  \label{p:conservative-descent-ff}
  Assume that we are in Setting~\ref{set:two-adjunctions} in the
  2-category of 
  categories (resp.\ triangulated categories). 
  Let $\alpha \colon L'G \xRa{\sim} FL$ be a
  2-isomorphism and assume that its mate 
  $\ol{\alpha} \colon GR \xRa{\sim} R'F$ is a 2-isomorphism as well.
  \begin{enumerate}
  \item
    \label{enum:L-ff}
    If the functor $G$ is
    conservative and the functor $L'$ is fully faithful, then $L$
    is fully faithful. 
  \item 
    \label{enum:R-ff}
    If $F$ is conservative and $R'$ is fully faithful, then
    $R$ is fully faithful.
  \end{enumerate}
  In particular, we get \emph{conservative descent for
    equivalences}: If both $F$ and $G$ are conservative and $(L',R')$
  is 
  an equivalence, then so is $(L,R)$.
\end{proposition}

\begin{proof}
  \ref{enum:L-ff}
  Observe that the commutative diagram
  \eqref{eq:unit-compatible-precursor}
  in Lemma~\ref{l:counits-compatible-precursor}
  yields
  the equality $(\ol{\alpha}L)
  \circ (G\eta)=(R'\alpha) \circ (\eta'G)$ of 2-morphisms, so
  that 
  $G\eta$ is a 2-isomorphism if and only if $\eta'G$ is a
  2-isomorphism.
  Now assume that $L'$ is full and faithful. This is equivalent to
  the condition that 
  $\eta'$ is a 2-isomorphism. Hence $\eta'G$ and $G\eta$ are
  2-isomorphisms. Since $G$ is conservative,
  this means that $\eta$ is a 2-isomorphism, i.e., $L$ is fully
  faithful.

  \ref{enum:R-ff}
  This follows similarly from the equality
  $(F\epsilon) \circ (\alpha R)= (\epsilon'F) \circ (L'\ol{\alpha})$
  obtained from the commutative diagram
  \eqref{eq:alphaR-L'olalpha}
  in Lemma~\ref{l:counits-compatible-precursor}.
\end{proof}

\begin{proposition}[Faithfully flat descent for fully
  faithfulness of Fourier--Mukai transforms]
  \label{p:FM-functor-ff-local}
  In Setting~\ref{set:FM-flat-BC}
  assume 
  that the flat base change morphism $u \colon S' \ra S$ is
  surjective, i.e.,
  faithfully flat. 
  Then the Fourier--Mukai transform
  $\Phi$ is
  fully faithful if its base change
  $\Phi'$ is fully faithful.
\end{proposition}

\begin{proof}
  Thanks to Proposition~\ref{p:flat-BC-FM}
  and Example~\ref{eg:conservative},
  Proposition~\ref{p:conservative-descent-ff}.\ref{enum:L-ff}
  applies.
\end{proof}
  
\subsection{Comonads and idempotent comonads}
\label{sec:comon-idemp-comon}

We continue to work in our fixed 2-category.
Standard references for monads and comonads are \cite[Section
VI]{maclane-working-mathematician}, \cite[Section
4]{borceux-cat-2} and~\cite{street-monads}.

\begin{definition}
  \label{d:comonad}
  A \define{comonad} is a quadruple $(\mathcal{C}, S, \epsilon,
  \delta)$ where $\mathcal{C}$ is an object, 
  $S \colon \mathcal{C} \ra \mathcal{C}$ is a 1-morphism, and
  \define{counit}
  $\epsilon \colon S \Ra \id_\mathcal{C}$ and \define{comultiplication} $\delta
  \colon S \Ra S^2$ are 
  2-morphisms such that the two diagrams
  \begin{equation}
    \label{eq:diagrams-comonad}
    \xymatrix{
      & {S} \ar@{=>}[ld]_-{\id} \ar@{=>}[d]^-{\delta}
      \ar@{=>}[rd]^-{\id}\\ 
      {S} & {S^2} \ar@{=>}[l]_-{\epsilon S}
      \ar@{=>}[r]^-{S\epsilon} & {S,} 
    } 
    \hspace{2em}
    \xymatrix{
      {S} \ar@{=>}[r]^-{\delta} \ar@{=>}[d]_-{\delta} 
      & {S^2} \ar@{=>}[d]^-{S\delta}\\ 
      {S^2} \ar@{=>}[r]^-{\delta S} 
      & {S^3} 
    }
  \end{equation}
  are commutative.

  A comonad
  $(\mathcal{C}, S, \epsilon, \delta)$ is
  \define{idempotent} if its comultiplication $\delta$ is a
  2-isomorphism 
  (see Remark~\ref{rem:define-idempotent-comonad} below for a shorter
  equivalent definition).

  By abuse of notation we often say that
  $S$ is a comonad on $\mathcal{C}$ or just that $S$ or
  $(S, \epsilon, \delta)$
  is a comonad.
\end{definition}

\begin{example}
  \label{eg:comonad-from-adjunction}
  Let
  $(L \colon
  \mathcal{D} \rla 
  \mathcal{C} \colon R, \eta, \epsilon)$ be an adjunction where $\eta
  \colon \id_\mathcal{D} 
  \Ra RL$ is the unit and $\epsilon \colon LR \Ra
  \id_\mathcal{C}$ is the counit. Then we obtain an associated comonad
  \begin{equation}
    \label{eq:assoc-comonad}
    (S:=LR, \epsilon, \delta:=L\eta R)
  \end{equation}
  on $\mathcal{C}$.
  It is idempotent if the unit $\eta$ is a 2-isomorphism.

  If we work in the 2-category of (triangulated) categories, then
  the functor  
  $L$ is fully faithful if and only if $\eta$ is a
  2-isomorphism. 
  In this case, the essential images of $L$ and $S=LR$ coincide,
  i.e., $\im L=\im S$,
  because $L\eta \colon L \Ra LRL=SL$ is a 2-isomorphism.
  In particular, 
  any fully faithful left adjoint functor $L$ (which is part of an
  adjunction 
  as above) 
  gives rise to an idempotent comonad with the same essential image.
\end{example}

\begin{example}
  \label{eg:FM-comonad}
  Let $\Phi = (K, p, q, \mathcal{K}) \colon X \ra Y$ be a relative
  Fourier--Mukai 
  transform. Assume that the induced functor $\Phi \colon \D_\qc(X)
  \ra \D_\qc(Y)$ is fully faithful, and denote its right adjoint
  by $\Psi$. Then 
  \begin{equation}
    \label{eq:FM-comonad}
    S_\Phi:=\Phi \Psi
  \end{equation}
  is an idempotent comonad on $\D_\qc(Y)$ whose essential
  image is the essential image of $\Phi$.
  This is a special case of Example~\ref{eg:comonad-from-adjunction}.
\end{example}

\begin{lemma}
  \label{l:idempotent-comonad}
  A comonad $(S, \epsilon, \delta)$ is idempotent if and only if
  $\epsilon S$ (resp.\ $S\epsilon$) is a 2-isomorphism.  
  If these conditions are satisfied then
  $\epsilon S = S \epsilon = \delta^{-1}$ and
  $\delta S= S\delta$.
\end{lemma}

\begin{proof}
  This follows from the defining commutative diagrams in
  \eqref{eq:diagrams-comonad}.
\end{proof}

\begin{lemma}
  \label{l:define-idempotent-comonad}
  Let $S \colon \mathcal{C} \ra \mathcal{C}$ be a 1-morphism and
  $\epsilon \colon S \Ra \id_\mathcal{C}$ a 2-morphism. If $S\epsilon$
  and $\epsilon S$ are equal  
  2-isomorphisms, then there is a unique $\delta \colon S \Ra
  S^2$, namely $\delta=(S\epsilon)^{-1}=(\epsilon S)^{-1}$, such
  that $(\mathcal{C}, S, \epsilon, \delta)$ is a comonad. 
  This comonad is idempotent. 
\end{lemma}

\begin{proof}
  The left diagram in \eqref{eq:diagrams-comonad}
  shows that we have to put
  $\delta=(S\epsilon)^{-1}=(\epsilon S)^{-1}$.
  But then $S\delta$ and $\delta S$ are invertible and
  $(S\delta)^{-1}=S\delta^{-1} = S \epsilon S =
    \delta^{-1}S =
    (\delta S)^{-1}$.
  This implies that the right diagram in
  \eqref{eq:diagrams-comonad} is commutative.
\end{proof}

\begin{remark}
  [Alternative definition of an idempotent comonad]
  \label{rem:define-idempotent-comonad}
  Giving an idempotent comonad $(\mathcal{C}, S, \epsilon, \delta)$ 
  is the same thing as giving a triple $(\mathcal{C}, S, \epsilon)$
  where 
  $S \colon \mathcal{C} \ra \mathcal{C}$ is a 1-morphism and 
  $\epsilon \colon S \Ra \id_\mathcal{C}$ 
  is a 2-morphism
  such that
  $S\epsilon=\epsilon S$ is a 2-isomorphism.  This is obvious
  from Lemmas~\ref{l:idempotent-comonad} and
  \ref{l:define-idempotent-comonad}. Hence it suffices to write
  $(\mathcal{C}, S,\epsilon)$ when referring to an idempotent comonad.
  Again, we often just write $(S,
  \epsilon)$ or $S$ for an idempotent comonad.
\end{remark}

\begin{remark}[Idempotent comonads versus colocalizations]
  \label{rem:ip-comonad-colocalization}
  An endofunctor $S \colon \mathcal{C} \ra \mathcal{C}$ of a
  (triangulated) category is a
  \emph{colocalization functor} 
  in the sense of \cite[2.4, 2.8]{krause-localization}
  if there exists a 2-morphism
  $\epsilon \colon S \Ra \id_\mathcal{C}$ such that $(S,
  \epsilon)$ is an idempotent comonad
  (use Remark~\ref{rem:define-idempotent-comonad}); if 
  $\epsilon' \colon S
  \Ra \id_\mathcal{C}$ is another 2-morphism turning $S$ into an
  idempotent comonad, 
  there is a unique 
  2-automorphism $\mu \colon S \xRa{\sim} S$ such that
  $\epsilon=\epsilon' \circ \mu$ (see
  \cite[Remark~2.5.5.(1)]{krause-localization} for the
  corresponding statement for localizations).
  In particular, the difference between the definition of an
  idempotent comonad and a colocalization functor consists in
  making $\epsilon$ part of the datum or not.
\end{remark}

\begin{remark}
  \label{rem:image-ip-comonad-right-admissible}
  If $(S, \epsilon)$ is an idempotent comonad on a (triangulated)
  category $\mathcal{C}$, then the 
  inclusion
  functor $\im S \ra \mathcal{C}$ admits a right adjoint, e.g.,
  the functor $\mathcal{C} \ra \im S$ induced by $S$.
\end{remark}

\subsection{Compatibilities of idempotent comonads}
\label{sec:comp-idemp-comon}

\begin{definition}
  \label{d:compatibility}
  Let $(\mathcal{C}, S, \epsilon)$ and $(\mathcal{C}', S',
  \epsilon')$ be idempotent comonads and
  let $F \colon
  \mathcal{C} \ra \mathcal{C}'$ be a 1-morphism.
  An \define{$F$-compatibility}
  from $S$ to $S'$, written $\sigma \colon S \xRa{F} S'$, is a
  2-isomorphism 
  $\sigma \colon FS \xRa{\sim} S'F$
  such that the diagram 
  \begin{equation}
    \label{eq:diagram-compatibility}
    \xymatrix{
      {FS} 
      \ar@{=>}[d]_-{F\epsilon}
      \ar@{=>}[r]^-{\sigma}_-{\sim} &
      {S'F} 
      \ar@{=>}[d]_-{\epsilon' F} \\
      {F} 
      \ar@{=}[r] &
      {F} 
    }
  \end{equation}
  commutes. If there exists at least one $F$-compatibility
  from $S$ to $S'$, 
  we say that $S$ and $S'$ are \emph{compatible with respect to $F$}.
\end{definition}

\begin{remark}
  Our notion of an $F$-compatibility is closely related to the
  notion of a morphism of comonads as defined in
  \cite{street-monads}):
  a morphism 
  $(F, \sigma) \colon (\mathcal{C}, S, \epsilon, \delta) \ra (\mathcal{C}',
  S', \epsilon', \delta')$ of comonads
  consists of a 1-morphism $F \colon \mathcal{C} \ra \mathcal{C}'$ 
  and a 2-morphism $\sigma \colon FS \Ra S'F$ 
  such that two obvious diagrams are commutative.
  
  If we assume that $S$ and $S'$ are idempotent comonads and 
  fix a 1-morphism $F \colon
  \mathcal{C} \ra \mathcal{C}'$, then the set
  of
  $F$-compatibilities coincides
  precisely with the set of 
  2-iso\-mor\-phisms $\sigma \colon FS \xRa{\sim} S'F$ such that
  $(F,\sigma)$ is a morphism of comonads. The proof of this
  result is not difficult and uses
  Lemma~\ref{l:define-idempotent-comonad}. We omit the details
  since we do not use arbitrary morphisms of comonads in the
  following.
\end{remark}

\subsection{Monads and compatibilities}
\label{sec:monads}

All results we have proven for (idempotent) comonads have analogs
for (idempotent) monads. We quickly review what we need.

A \define{monad} is a quadruple $(\mathcal{D}, T, \eta, \mu)$
where $\mathcal{D}$ is an object, $T \colon \mathcal{D} \ra
\mathcal{D}$ is a 1-morphism, and \define{unit} $\eta \colon
\id_\mathcal{D} \ra T$ and \define{multiplication} $\mu \colon
T^2 \Ra T$ are 2-isomorphisms making the obvious two diagrams
commutative; the shape of these two diagrams is obtained from the
two diagrams in
\eqref{eq:diagrams-comonad} 
by reversing the direction of all 2-morphisms.
Such a monad is \define{idempotent} if its
multiplication $\mu$ is a 2-isomorphism.

An idempotent monad is equivalently given by a triple 
$(\mathcal{D}, T, \eta)$ where $T\eta$ and $\eta T$ are equal and
invertible (by the monadic version of
Remark~\ref{rem:define-idempotent-comonad}). 

\begin{example}
  [cf.\ Example~\ref{eg:comonad-from-adjunction}]
  \label{eg:monad-from-adjunction}
  Let
  $(L \colon
  \mathcal{D} \rla 
  \mathcal{C} \colon R, \eta, \epsilon)$ be an adjunction where $\eta
  \colon \id_\mathcal{D} 
  \Ra RL$ is the unit and $\epsilon \colon LR \Ra
  \id_\mathcal{C}$ is the counit. Then we obtain an associated monad
  \begin{equation}
    \label{eq:assoc-monad}
    (T:=RL, \eta, \mu:=R\epsilon L)
  \end{equation}
  on $\mathcal{C}$.
  It is idempotent if the counit $\epsilon$ is a 2-isomorphism.
\end{example}

Given idempotent monads $(\mathcal{D}, T, \eta)$ and
$(\mathcal{D}', T', \eta')$ and a 1-morphism
$G \colon \mathcal{D} \ra \mathcal{D}'$, a \define{$G$-compatibility}
from $T$ to $T'$, written $\tau \colon T \xRa{G} T'$, is
a 2-isomorphism $\tau \colon GT \xRa{\sim} T'G$ such that the diagram
\begin{equation}
  \xymatrix{
    {GT} 
    \ar@{=>}[r]^-{\tau}_-{\sim} &
    {T'G} \\
    {G} 
    \ar@{=>}[u]_-{G\eta}
    \ar@{=}[r] &
    {G} 
    \ar@{=>}[u]_-{\eta' G}
  }
\end{equation}
is commutative. 
The direction of $\tau$ is chosen with regard to our later
applications. 

\subsection{Compatibilities from mates}
\label{sec:comp-from-mates}

\begin{lemma}
  \label{l:idempotent-comonads-compatible}
  In Setting~\ref{set:two-adjunctions} assume that the units
  $\eta$ and $\eta'$ of
  both adjunctions $(L, R, \eta, \epsilon)$ and
  $(L', R', \eta', \epsilon')$ are 2-isomorphisms.
  Assume that there is a 2-isomorphism $\alpha \colon L'G \Ra FL$
  whose mate $\ol{\alpha}$ is a 2-isomorphism. Then 
  \begin{equation}
    \label{eq:compatibility-comonads}
    (L'R'F\epsilon) \circ (\epsilon' FLR)^{-1}
    \colon 
    (S=LR, \epsilon) \xRa{F} (S'=L'R', \epsilon')
  \end{equation}
  is an $F$-compatibility between the associated 
  idempotent comonads.
\end{lemma}

\begin{proof}
  Since $\eta$ and $\eta'$ are 2-isomorphisms,
  the associated comonads
  $S$
  and $S'$ are idempotent by
  Example~\ref{eg:comonad-from-adjunction}.
  Consider the commutative diagram
  \eqref{eq:alphaR-L'olalpha}
  of Lemma~\ref{l:counits-compatible-precursor}.
  Invertibility of $\alpha$, $\ol\alpha$ and $\eta'
  G$ shows that
  the two 2-morphisms $\epsilon'FLR$
  and $L'R'F\epsilon$ are invertible.
  Hence $(L'R'F\epsilon) \circ (\epsilon'FLR)^{-1}$ is a
  2-isomorphism whose composition with $\epsilon'F$ is
  $F\epsilon$. This just means that it is an $F$-compatibility
  from
  the idempotent comonad $S$ to the idempotent comonad $S'$.
\end{proof}

\begin{proposition}
  \label{p:FM-comonads-f*-compatible}
  In Setting~\ref{set:FM-flat-BC}
  assume that the Fourier--Mukai transform
  $\Phi$ and its base change $\Phi'$ are fully faithful. Then the
  associated 
  idempotent comonads $S_\Phi$ on $\D_\qc(Y)$ and $S_{\Phi'}$ on $\D_\qc(Y')$ 
  (see Example~\ref{eg:FM-comonad})
  are compatible with respect to $g^* \colon \D_\qc(Y) \ra
  \D_\qc(Y')$.
\end{proposition}

\begin{proof}
  This is a consequence of Proposition~\ref{p:flat-BC-FM} and
  Lemma~\ref{l:idempotent-comonads-compatible}.
\end{proof}


\section{Semi-orthogonal decompositions and descent}
\label{sec:semi-decomp-desc}

The main goal of this section is 
the conservative descent Theorem~\ref{t-main-categorical}.

\subsection{Semi-orthogonal decompositions}
\label{sec:semi-orth-decomp}

General results on admissible subcategories and
semi-orthogonal 
decompositions of triangulated categories can be found in
\cite{bondal-kapranov-representable-functors} and
\cite[Appendix~A]{valery-olaf-matfak-semi-orth-decomp}.
We recall the basic definitions here.

\begin{definition}
Let $\mathcal{T}$ be a triangulated category.
A \define{right} (resp.\ \define{left}) \define{admissible
subcategory of} $\mathcal{T}$ 
is a strictly full triangulated subcategory $\mathcal{U}$ of $\mathcal{T}$
such that the inclusion functor 
$\mathcal{U} \to \mathcal{T}$ admits a right (resp.\ left) adjoint.
An \define{admissible} subcategory is a subcategory that is both
left and right admissible.
\end{definition}

We remind the reader that an adjoint functor of a triangulated
functor is triangulated (in a canonical way if the adjunction is
fixed), 
see e.g. \cite[\sptag{0A8D}]{stacks-project}.

\begin{example}
  \label{eg:id-comonad-right-admissible-subcat}
  Let $(S, \epsilon)$ be an idempotent comonad on a triangulated
  category $\mathcal{T}$. Then $\im S$ is a
  right admissible subcategory of $\mathcal{T}$, by
  Remark~\ref{rem:image-ip-comonad-right-admissible}.
  Any right admissible subcategory is of this form (by
  Example~\ref{eg:comonad-from-adjunction}). 
  Similarly, the essential image of an idempotent monad is a left
  admissible subcategory, and any left admissible subcategory is of
  this form.
\end{example}

\begin{definition}
\label{d:semi-orth}
Let $\mathcal{T}$ be a triangulated category.
A sequence $\mathcal{T}_1, \ldots, \mathcal{T}_n$ of strictly
full triangulated
subcategories of 
$\mathcal{T}$ is called \define{semi-orthogonal} if
$\mathcal{T}(A_i, A_j) = 0$ for all objects
$A_i \in \mathcal{T}_i$ and $A_j \in \mathcal{T}_j$ whenever $i >
j$. It is called \define{full} (in $\mathcal{T}$) if
$\mathcal{T}$ coincides with the smallest
strictly full triangulated subcategory of $\mathcal{T}$
that contains all the categories $\mathcal{T}_i$.
A \define{semi-orthogonal decomposition} of $\mathcal{T}$
is a full semi-orthogonal sequence 
$\mathcal{T}_1, \ldots, \mathcal{T}_n$
and is denoted as 
\begin{equation*}
  \mathcal{T} = \langle \mathcal{T}_1, \ldots, \mathcal{T}_n \rangle.
\end{equation*}
\end{definition}

\subsection{Complementary idempotent comonads and monads}
\label{sec:complementary-monads}

In the rest of this section, we always work in the 2-category of
triangulated categories. In particular, if we say that
$(\mathcal{T}, S, \epsilon, \delta)$ is a comonad then
$\mathcal{T}$ is a triangulated category, $S$ is a 
triangulated functor and $\epsilon$ and $\delta$ are
morphisms of triangulated functors.

\begin{definition}
  \label{d:complementary}
  Let $\mathcal{T}$ be a triangulated category.
  We say that an idempotent comonad $(S,\epsilon)$ and an
  idempotent monad $(T,\eta)$ on $\mathcal{T}$ are
  \define{complementary} if for any object $A \in \mathcal{T}$
  there is a morphism 
  $\partial_A$ such that 
  \begin{equation}
    \label{eq:perp-triangle}
    SA \xra{\epsilon_A} A \xra{\eta_A} T A \xra{\partial_A}
    \Sigma SA
  \end{equation}
  is a triangle.
  We also just say that $T$ is complementary to $S$ or that $S$
  is complementary to $T$. 
\end{definition}

\begin{lemma}
  \label{l:images-perp-(co)monads}
  If an idempotent comonad $S$ and an idempotent monad $T$
  on a triangulated category $\mathcal{T}$ are
  complementary, then 
  \begin{align}
   (\im S)^\perp & = \im T = \ker S,\\
    \leftidx{^\perp}{(\im T)}{} & = \im S = \ker T.
  \end{align}
  In particular, $\im S$ and $\im T$ are closed under direct
  summands in $\mathcal{T}$.
  Moreover, the morphism
  $\partial_A$ 
  in triangle \eqref{eq:perp-triangle} is unique,
  and mapping $A$ to the triangle \eqref{eq:perp-triangle}
  extends uniquely to a functor from $\mathcal{T}$ to the
  category of 
  triangles in $\mathcal{T}$.
\end{lemma}

\begin{proof}
  We claim that $\mathcal{T}(\im S, \im T)=0$.
  Let $f \colon SA \ra TB$ be a morphism where $A$ and $B$ are objects
  of $\mathcal{T}$. Consider the following commutative diagram
  \begin{equation}
    \xymatrix{
      {S^2A} \ar[r]^-{\epsilon_{SA}} &
      {SA} \ar[r]^-{\eta_{SA}} \ar[d]_-{f} &
      {TSA} \ar[d]_-{Tf}\\
      &
      {TB} \ar[r]^-{\eta_{TB}} &
      {TTB.}
    }
  \end{equation}
  Its upper row can be completed to a triangle, and
  $\epsilon_{SA}$ and $\eta_{TB}$ are isomorphisms. Hence
  $TSA=0$ and $f=0$. This proves the claim, and the four
  equalities in the lemma follow immediately.
  Uniqueness of $\partial_A$ follows from
  \cite[Corollaire~1.1.10]{BBD}. Functoriality of the
  triangle follows from \cite[Corollaire~1.1.9]{BBD}.
\end{proof}

\begin{example}
  \label{eg:complementary->sod}
  Given an idempotent comonad $S$ and a complementary
  idempotent monad $T$ on a triangulated category
  $\mathcal{T}$, there is semi-orthogonal decomposition
  $\mathcal{T}=\langle \im T, \im S\rangle$. Any
  semi-orthogonal decomposition with two components arises in this
  way.
\end{example}

\begin{remark}[Existence of complementary comonads and monads]
  \label{rem:existence-perp-monad}
  Any idempotent comonad (resp.\ monad) on
  a triangulated category $\mathcal{T}$ admits
  a complementary idempotent monad (resp.\ comonad).
  This follows from well-known arguments, e.g., see
  \cite[Proposition~4.12.1]{krause-localization} using
  Remark~\ref{rem:ip-comonad-colocalization} and
  the corresponding statement for idempotent monads.
  Given an idempotent comonad $(S, \epsilon)$,
  the idea is to 
  complete for each object $A \in \mathcal{T}$ the morphism
  $\epsilon_A$ to a triangle~\eqref{eq:perp-triangle}, and to
  show that this extends to morphisms and in fact defines an
  idempotent monad $(T, \eta)$ in the 2-category of triangulated
  categories.
\end{remark}

\begin{remark}[Uniqueness of complementary comonads and monads]
  \label{rem:uniqueness-perp-monad}
  If an idempotent comonad $(S, \epsilon)$ on $\mathcal{T}$
  admits two complementary idempotent monads $(T, \eta)$ and
  $(T', \eta')$, then there is a unique 2-isomorphism $\nu \colon T
  \xRa{\sim} T'$ such $\nu \circ \eta =\eta'$. This follows from
  \cite[Corollaire 1.1.9]{BBD}
  or Proposition~\ref{p:complementary-monads-compatible}
  below for $\mathcal{T}=\mathcal{T'}$ and $F=\id_\mathcal{T}$.
  The dual statement is also true.
\end{remark}

\begin{remark}
  \label{rem:notation-perp-monad}
  Given an idempotent comonad $(S, \epsilon)$ on $\mathcal{T}$ we
  often use the notation $(S^\perp, \eta)$ for a
  complementary monad. Such a complementary monad always exists
  (Remark~\ref{rem:existence-perp-monad}) and is as unique as
  possible (Remark~\ref{rem:uniqueness-perp-monad}).
  The notation $S^\perp$ is motivated by the equality
  $(\im S)^\perp=\im (S^\perp)$, see Lemma~\ref{l:images-perp-(co)monads}.
\end{remark}

\subsection{Semi-orthogonal decompositions and idempotent comonads}
\label{sec:semi-decomp-and-comonads}

Given a sequence of idempotent comonads, we reformulate the
defining conditions that their essential images form a semi-orthogonal
decomposition in terms of vanishing conditions on certain
compositions of these comonads and their complementary monads.
This reformulation is a key ingredient in the proof of
the conservative descent Theorem~\ref{t-main-categorical} below.

\begin{proposition}
  \label{prop-semi-orthogonal}
  Let
  $S_1, S_2, \ldots, S_n$ be a sequence of idempotent
  comonads on a triangulated category $\mathcal{T}$.  
  Then the following statements
  are equivalent.
  \begin{enumerate}
  \item The sequence $\im S_1, \ldots, \im S_n$ of essential
    images is semi-orthogonal.
  \item The composition $S_iS_j$ vanishes for all $i > j$.
  \end{enumerate}
\end{proposition}

\begin{proof}
  It is certainly enough to prove the statement in case $n=2$.
  So we need to prove that $\mathcal{T}(\im S_2, \im S_1)=0$ if and
  only if $S_2S_1=0$.

  The condition $\mathcal{T}(\im S_2, \im S_1)=0$
  is equivalent to $\im S_1 \subset (\im S_2)^\perp$.
  Lemma~\ref{l:images-perp-(co)monads}
  shows
  $(\im S_2)^\perp=\ker S_2$ since
  $S_2$ has a complementary monad
  (see Remark~\ref{rem:existence-perp-monad}).
  Finally, the condition $S_2S_1=0$ is equivalent to 
  $\im S_1 \subset \ker S_2$.
  %
  %
  %
\end{proof}

\begin{proposition}
\label{prop-decomposition}
Let $S_1, S_2, \ldots, S_n$ be a sequence of idempotent
comonads on a triangulated category $\mathcal{T}$, with complementary idempotent monads $S_1^\perp, \dots, S_n^\perp$. 
Assume that the sequence $\im S_1, \ldots, \im S_n$ is
semi-orthogonal. Then the following statements are equivalent:
\begin{enumerate}
\item
The sequence $\im S_1, \ldots, \im S_n$ is full, i.e., it
forms a semi-orthogonal decomposition
of~$\mathcal{T}$.
\item
The composition $S_1^\perp S_2^\perp \cdots S_n^\perp$ vanishes on
$\mathcal{T}$.
\end{enumerate}
\end{proposition}

\begin{proof}
  Let $\mathcal{T}'$ denote the triangulated hull of $\im S_1,
  \ldots, \im S_n$ in $\mathcal{T}$.

Assume that $S_1^\perp S_2^\perp \cdots S_n^\perp$ vanishes.
We want to show that any object $A$ in $\mathcal{T}$ lies
in $\mathcal{T}'$. Diagram~\eqref{eq:cofiltration} below
illustrates the following proof in case $n=3$.
By assumption, we have $S_1^\perp \dots S_n^\perp A =0$, and this
object trivially lies in
$\mathcal{T}'$. By induction over $i \in \{1, \dots, n\}$
assume that $S_i^\perp \cdots S_n^\perp A$
lies in $\mathcal{T}'$ and consider
the triangle
\begin{equation}
  \label{eqn-tri}
  S_iS_{i+1}^\perp \cdots S_n^\perp A \xra{\epsilon_i}
  S_{i+1}^\perp \cdots S_n^\perp A \xra{\eta_i}
  S_i^\perp S_{i+1}^\perp \cdots S_n^\perp A \to.
\end{equation}
Since the first object lies in $\im S_i \subset \mathcal{T}'$, we
deduce that also
$S_{i+1}^\perp \cdots S_n^\perp A$ lies in~$\mathcal{T}'$. So
eventually we get $A \in \mathcal{T}'$.
This shows that $\mathcal{T}'=\mathcal{T}$, i.e., the sequence
$\im S_1, \ldots, \im S_n$ is full.
Note that semi-orthogonality was in fact not used for this implication.

Conversely, assume that $\mathcal{T} = \mathcal{T}'$. 
We need to show $S_1^\perp S_2^\perp\cdots
S_n^\perp A=0$ for any object $A \in \mathcal{T}$. It is enough to show that 
$\mathcal{T}(B, S_1^\perp\cdots S_n^\perp A) = 0$
for any $j \in \{1, \dots, n\}$ and
any $B \in \im S_j$. Fix such $j$ and $B$. 

We prove that $\mathcal{T}(B, S_i^\perp\cdots S_n^\perp A) = 0$ by
descending induction over 
$i \in \{1, \ldots, j\}$. The case $i=j$ is obvious since $(\im
S_j^\perp)=(\im S_j)^\perp$.
Assume that $i < j$ and that we already know that 
$\mathcal{T}(B, S_{i+1}^\perp\cdots S_n^\perp A) = 0$.
Applying $\mathcal{T}(B, -)$ to the triangle \eqref{eqn-tri} and
using semi-orthogonality $\mathcal{T}(\im S_j, \im S_i)=0$ we obtain the
isomorphism 
\begin{equation}
  0 =
  \mathcal{T}(B, S_{i+1}^\perp\cdots S_n^\perp A)
  \sira \mathcal{T}(B, S_i^\perp S_{i+1}^\perp\cdots S_n^\perp A).
\end{equation}
By induction, this proves what we need.
\end{proof}

\begin{example}
  \label{eg:proof-n=3}
  The following diagram illustrates the first argument of the
  above proof in case $n=3$.
  It shows that any object $A$ can be written as an iterated extension of
  an object of $\im S_1$ by an object of $\im S_2$ by an object
  of $\im S_3$.
  \begin{equation}
    \label{eq:cofiltration}
    \xymatrix{
      {A} \ar[r]^-{\eta_3} &
      {S_3^\perp A} \ar[r]^-{\eta_2} \ar@{~>}[dl] &
      {S_2^\perp S_3^\perp A} \ar[r]^-{\eta_1} \ar@{~>}[dl] &
      {S_1^\perp S_2^\perp S_3^\perp A} \ar@{}[r]|-{=}
      \ar@{~>}[dl] &
      {0}\\
      {S_3A} \ar[u]^-{\epsilon_3} &
      {S_2 S_3^\perp A} \ar[u]^-{\epsilon_2} &
      {S_1 S_2^\perp S_3^\perp A} \ar[u]^-{\epsilon_1}_-{\sim}
    }
  \end{equation}
\end{example}

\subsection{Compatibilities and complementary comonads and monads}
\label{sec:comp-and-perp}

\begin{proposition}
  \label{p:complementary-monads-compatible}
  Let $(S, \epsilon)$ and $(S', \epsilon')$ be idempotent
  comonads on triangulated categories $\mathcal{T}$ and
  $\mathcal{T}'$, 
  and let
  $(S^\perp, \eta)$ and $(S'^\perp, \eta')$
  be complementary idempotent monads.
  Let $F \colon \mathcal{T} \ra \mathcal{T}'$ be a 
  triangulated functor. 
  Given any $F$-compatibility 
  \begin{equation}
    \sigma \colon S \xRa{F} S'
  \end{equation}
  of comonads there is a unique
  $F$-compatibility  
  \begin{equation}
    \sigma^\perp \colon S^\perp \xRa{F} S'^\perp
  \end{equation}
  of monads such that 
  \begin{equation}
    \xymatrix{
      {FS} \ar@{=>}[r]^-{F\epsilon} \ar@{=>}[d]_-{\sigma}^-{\sim}
      & 
      {F} \ar@{=>}[r]^-{F\eta} \ar@{=>}[d]_-{\id_F}^-{\sim} &
      {FS^\perp} \ar@{=>}[r] \ar@{=>}[d]_-{\sigma^\perp}^-{\sim} &
      {\Sigma FS} \ar@{=>}[d]_-{\Sigma\sigma}^-{\sim}\\
      {S'F} \ar@{=>}[r]^-{\epsilon'F} &
      {F} \ar@{=>}[r]^-{\eta'F} &
      {S'^\perp F} \ar@{=>}[r] &
      {\Sigma S'F} 
    }
  \end{equation}
  is a functorial isomorphism of triangles, i.e.\ plugging in any
  object $A \in \mathcal{T}$ yields an isomorphism of triangles in
  $\mathcal{T}'$, and these morphisms are 
  compatible with morphisms in $\mathcal{T}$. 
\end{proposition}

\begin{remark}
  The map
  $\sigma \mapsto \sigma^\perp$ defines in fact a
  bijection between the set of $F$-compatibilities from $S$ to
  $S'$ and the set of $F$-compatibilities from $S^\perp$ to
  $S'^\perp$.
\end{remark}

\begin{proof}
  Let $\sigma$ be an $F$-compatibility as above.
  For any $A \in \mathcal{T}$ complete the partial morphism
  \begin{equation}
    \xymatrix{
      {FSA} \ar[r]^-{F\epsilon_A} \ar[d]_-{\sigma_A}^-{\sim} &
      {FA} \ar[r]^-{F\eta_A} \ar@{=}[d] &
      {FS^\perp A} \ar[r] \ar@{..>}[d]_-{\sigma^\perp_A} &
      {\Sigma FSA} \ar[d]_-{\Sigma\sigma_A}^-{\sim}\\
      {S'FA} \ar[r]^-{\epsilon'_{FA}} &
      {FA} \ar[r]^-{\eta'_{FA}} &
      {S'^\perp FA} \ar[r] &
      {\Sigma S'FA} 
    }
  \end{equation}
  of triangles by the dotted arrow to an (iso)morphism of
  triangles. 
  This dotted arrow is already uniquely specified by the
  requirement that the square in the middle is commutative, by 
  \cite[Corollaire 1.1.10]{BBD}, since the
  object $S'^\perp FA$ lies in $\im (S'^\perp)=(\im S')^\perp$
  (by Lemma~\ref{l:images-perp-(co)monads}),
  and the objects 
  $FSA$ 
  and $\Sigma FSA$ lie in $\im S'$ since $\sigma_A$ is an
  isomorphism.
  Similarly, the morphisms $(\sigma_A, \id_{FA}, \sigma^\perp_A)$ of
  triangles, for $A \in \mathcal{T}$, 
  are easily seen to be compatible with morphisms $A \ra
  A'$ in $\mathcal{T}$. In particular, $\sigma^\perp \colon
  FS^\perp \xRa{A} S'^\perp F$ is a 2-isomorphism between
  functors, and it is easy to see that it is compatible with
  suspensions, i.e., it is a 2-isomorphism between triangulated
  functors. 
  This proves the proposition.
\end{proof}

\begin{example}
  \label{eg:FM-perp-monads-f*-compatible}
  The conclusion of
  Proposition~\ref{p:FM-comonads-f*-compatible} can be extended:
  the complementary idempotent monads $S_\Phi^\perp$ and
  $S_{\Phi'}^\perp$ are also $g^*$-compatible, by 
  Proposition~\ref{p:complementary-monads-compatible}.
\end{example}

\subsection{Conservative descent}
\label{sec:conservative-descent}

\begin{theorem}[Conservative descent]
\label{t-main-categorical}
\label{t:descent-comonadic-sod}
Let $F \colon \mathcal{T} \ra \mathcal{T}'$ be a conservative 
triangulated functor.
Let $S_1, \dots, S_n$ and $S_1', \dots, S_n'$ be sequences of idempotent comonads on $\mathcal{T}$ and $\mathcal{T}'$,
respectively,
and assume that $S_i$ and $S_i'$ are compatible
with respect to $F$ for each~$i$.
If the sequence
\begin{equation}
\label{eq:sod-upstairs}
\im S'_1, \ldots, \im S'_n
\end{equation}
of essential images is semi-orthogonal in $\mathcal{T}'$,
then so is the sequence
\begin{equation}
\label{eq:sod-downstairs}
\im S_1, \ldots, \im S_n
\end{equation}
in $\mathcal{T}$. Both sequences consist of right admissible
subcategories of $\mathcal{T}'$ and $\mathcal{T}$, respectively. 
Moreover, if \eqref{eq:sod-upstairs} is a semi-orthogonal decomposition of $\mathcal{T}'$,
then \eqref{eq:sod-downstairs} is a semi-orthogonal decomposition of $\mathcal{T}$.
\end{theorem}
\begin{proof}
  By Example~\ref{eg:id-comonad-right-admissible-subcat}, 
  both sequences consist of right admissible
  subcategories of $\mathcal{T}'$ and $\mathcal{T}$, respectively.

  Choose $F$-compatibilities 
  $S_i \xRa{F} S_i'$. They are
  given by 2-isomorphisms
  $\sigma_i \colon F S_i \xRa{\sim} S_i'F$.
  We obtain 2-isomorphisms
  \begin{equation}
    FS_iS_j \xRa[\sim]{\sigma_i S_j}
    S_i'FS_j \xRa[\sim]{S_i'\sigma_j}
    S_i'S_j'F.
  \end{equation}
  We use Proposition~\ref{prop-semi-orthogonal} twice.
  Semi-orthogonality of $\im S_1', \dots, \im S_n'$
  gives $S_i'S_j'=0$ for $i>j$
  and hence $FS_iS_j=0$. Since $F$ is conservative
  we deduce $S_iS_j=0$
  (Example~\ref{eg:conservative}), so 
  $\im S_1, \dots, \im S_n$ is semi-orthogonal.

  Proposition~\ref{p:complementary-monads-compatible} provides
  $F$-compatibilities $S_i^\perp \xRa{F} S_i'^\perp$ given by 
  2-isomorphisms
  $\sigma_i^\perp \colon FS_i^\perp \xRa{\sim} S_i'^\perp F$.
  We obtain 2-isomorphisms
  \begin{equation}
    FS_1^\perp S_2^\perp \dots S_n^\perp
    \xRa[\sim]{\sigma^\perp_1 S_2^\perp \dots S_n^\perp}
    S_1'^\perp F S_2^\perp \dots S_n^\perp
    \xRa[\sim]{}
    \dots
    \xRa[\sim]{}
    S_1'^\perp S_2'^\perp \dots S_n'^\perp F.
  \end{equation}
  Assume that \eqref{eq:sod-upstairs} is a semi-orthogonal
  decomposition of $\mathcal{T}'$.
  Then the expression on the right vanishes by 
  Proposition~\ref{prop-decomposition}.
  Hence the expression on the left vanishes and 
  we get 
  $S_1^\perp S_2^\perp \cdots S_n^\perp=0$
  because $F$ is conservative. 
  Proposition~\ref{prop-decomposition}
  then shows that \eqref{eq:sod-downstairs} is a semi-orthogonal
  decomposition of $\mathcal{T}$ since we already know that it is
  a semi-orthogonal sequence.
\end{proof}

\subsection{Induced semi-orthogonal decompositions}
\label{sec:induc-semi-decomp}

Let $\mathcal{T}'$ be a strictly full triangulated subcategory of a
triangulated category $\mathcal{T}$.
Let $(S, \epsilon)$ be an idempotent comonad on
$\mathcal{T}$ such that 
$S$ restricts to a functor $S' \colon \mathcal{T}' \ra \mathcal{T}'$.
Then we
obtain by restriction an idempotent comonad $(S', \epsilon')$ on
$\mathcal{T}'$.  
Moreover, the universal property of the
Verdier quotient 
$Q \colon \mathcal{T} \ra \mathcal{T}/\mathcal{T}'$
shows that there is a unique triangulated 
functor $S'' \colon \mathcal{T}/\mathcal{T}' \ra \mathcal{T}/\mathcal{T}'$
satisfying $S'' Q = Q S$. The 2-morphism $\epsilon$ descends similarly and we obtain an idempotent comonad
$(S'', \epsilon'')$ on $\mathcal{T}/\mathcal{T}'$. 

\begin{lemma}
  \label{l:idempotent-comonad-orthogonal-restrict}
  Let $S$ be an idempotent comonad on a triangulated
  category $\mathcal{T}$. Let $\mathcal{T}'$ be a strictly full
  triangulated subcategory of $\mathcal{T}$ such that $S$ preserves
  $\mathcal{T}'$. Denote the induced idempotent comonads on
  $\mathcal{T}'$ and $\mathcal{T}/\mathcal{T}'$ by $S'$ and
  $S''$, respectively. 
  Let $S^\perp$ be an idempotent monad which is complementary to
  $S$. Then $S^\perp$
  restricts to an
  idempotent monad $(S^\perp)'$ which is complementary to $S'$
  and induces an 
  idempotent monad $(S^\perp)''$ which is complementary to
  $S''$, i.e., 
  $(S^\perp)'=S'^\perp$ and $(S^\perp)''=(S'')^\perp$. 
\end{lemma}

\begin{proof}
  For any object $A$ of $\mathcal{T}$ we have the triangle
  $SA \xra{\epsilon_A} A \xra{\eta_A} S^\perp A \xra{\partial_A}
  \Sigma SA$. 
  If $A$ lies in $\mathcal{T}'$ then 
  $SA \in \mathcal{T}'$ and hence $S^\perp A \in \mathcal{T}'$ since
  $\mathcal{T}'$ is closed 
  under isomorphisms in $\mathcal{T}$. 
  This implies that $(S^\perp, \eta)$ restricts to an idempotent
  monad $((S^\perp)', \eta')$ on 
  $\mathcal{T}'$ and that it descends uniquely to an idempotent
  monad $((S^\perp)'', \eta'')$ on
  $\mathcal{T}/\mathcal{T}'$. 
  We obtain triangles
  \begin{equation}
    S'A' \xra{\epsilon'_{A'}} A'
    \xra{\eta'_{A'}} (S^\perp)'A' 
    \xra{\partial'_{A'}} 
    \Sigma S'A'
  \end{equation}
  in $\mathcal{T}'$ for $A' \in
  \mathcal{T}'$ and
  \begin{equation}
    S''A'' \xra{\epsilon''_{A''}} A''
    \xra{\eta''_{A''}} (S^\perp)'' A''
    \xra{\partial''_{A''}} 
    \Sigma S'' A''  
  \end{equation}
  in 
  $\mathcal{T}/\mathcal{T}'$ for
  $A'' \in \mathcal{T}/\mathcal{T}'$. From this, the lemma follows.
\end{proof}

\begin{proposition}
  \label{p:restrict-comonadic-sod}
  Let
  $S_1, \ldots, S_n$ be a sequence of idempotent
  comonads on a triangulated category $\mathcal{T}$.
  Let
  $\mathcal{T}' \subset \mathcal{T}$ be a strictly full triangulated
  subcategory such that all $S_i$ restrict to $\mathcal{T}'$.
  Denote
  the restricted idempotent comonads on $\mathcal{T}'$
  by $S_i'$ and the induced idempotent comonads on
  $\mathcal{T}/\mathcal{T}'$ by $S''_i$.
  If the sequence
  \begin{align}
    \label{eq:sequ-T}
    &\im S_1, \dots, \im S_n\\
    \intertext{of essential images is semi-orthogonal in
    $\mathcal{T}$, then so are the sequences}
    \label{eq:sequ-T'}
    &\im S_1', \dots, \im S_n',
    \\
    \label{eq:sequ-T/T'}
    &\im S''_1, \dots, \im S''_n
  \end{align}
  in $\mathcal{T}'$ and~$\mathcal{T}/\mathcal{T}'$, respectively.
  Moreover, if \eqref{eq:sequ-T} is a semi-orthogonal decomposition of
  $\mathcal{T}$, then
  \eqref{eq:sequ-T'} and~\eqref{eq:sequ-T/T'}
  are semi-orthogonal decompositions of
  $\mathcal{T}'$ and $\mathcal{T}/\mathcal{T}'$, respectively.
\end{proposition}

\begin{proof}
  The equality $S_iS_j=0$ implies the equalities
  $S_i'S_j'=0$ and $S''_iS''_j=0$, for all $i, j$.
  Lemma~\ref{l:idempotent-comonad-orthogonal-restrict} shows that
  $(S_i^\perp)'=(S_i')^\perp$ and $(S_i^\perp)''=(S''_i)^\perp$.
  Therefore the equality $S_1^\perp S_2^\perp \cdots S_n^\perp =0$
  implies the equalities $(S'_1)^\perp (S'_2)^\perp \cdots (S'_n)^\perp =0$
  and $(S''_1)^\perp (S''_2)^\perp \cdots (S''_n)^\perp =0$.
  Now use Propositions~\ref{prop-semi-orthogonal} and
  \ref{prop-decomposition}. 
\end{proof}

\begin{corollary}
  \label{c:restrict-sod-right-admissible}
  Let $L_i \colon \mathcal{T}_i \ra \mathcal{T}$ be fully
  faithful 
  triangulated functors admitting right adjoints $R_i$, for $i=1,
  \dots, n$. 
  Let $\mathcal{T}' \subset \mathcal{T}$ 
  and $\mathcal{T}_i' \subset \mathcal{T}_i$ be  
  strictly full triangulated subcategories such that all $L_i$
  and all $R_i$ restrict to functors
  $L_i' \colon \mathcal{T}'_i \ra \mathcal{T}'$ and
  $R_i' \colon \mathcal{T}' \ra \mathcal{T}'_i$, respectively.
  Then the $L'_i$ are fully faithful 
  triangulated functors and descend to
  fully faithful 
  triangulated functors
  $L''_i \colon \mathcal{T}_i/\mathcal{T}'_i \ra
  \mathcal{T}/\mathcal{T}'$, and all functors $L'_i$ and $L''_i$
  admit right adjoints.  
  Moreover, if the sequence
  \begin{equation}
    \label{eq:sequ-T-L}
    \im L_1, \dots, \im L_n
  \end{equation}
  is semi-orthogonal in $\mathcal{T}$, then so are the sequences
  \begin{align}
    \label{eq:sequ-T'-L}
    \im L_1', \dots, \im L_n',
    \\
    \label{eq:sequ-T/T'-L}
    \im L''_1, \dots, \im L''_n
  \end{align}
  in $\mathcal{T}'$ and~$\mathcal{T}''$, respectively.
  Furthermore, if \eqref{eq:sequ-T-L} is a semi-orthogonal
  decomposition of 
  $\mathcal{T}$, then
  \eqref{eq:sequ-T'-L} and
  \eqref{eq:sequ-T/T'-L}~are semi-orthogonal decompositions of
  $\mathcal{T}'$ and $\mathcal{T}/\mathcal{T}'$, respectively.
\end{corollary}

\begin{proof}
  Clearly, $L_i'$ is 
  fully faithful triangulated and has $R'_i$ as a right adjoint. The
  associated 
  idempotent comonad 
  $S_i'=L_i'R_i'$ satisfies 
  $\im L_i'=\im S_i'$ by
  Example~\ref{eg:comonad-from-adjunction}.
  The universal property of the Verdier quotient shows that
  $L_i$ descends to a triangulated functor  
  $L''_i \colon \mathcal{T}_i/\mathcal{T}'_i \ra
  \mathcal{T}/\mathcal{T}'$. 
  More precisely, the adjunction $(L_i, R_i, \eta_i, \epsilon_i)$ 
  descends to an adjunction
  $(L''_i, R''_i, \eta''_i, \epsilon''_i)$, so $L''_i$ has a
  right adjoint.
  Since $L_i$ is full and faithful, $\eta_i
  \colon \id \xRa{\sim} R_iL_i$ is an isomorphism, and hence so is
  $\eta''_i$. This shows that
  $L''_i$ is full and faithful. As above, the associated idempotent
  comonad $S''_i = L''_i R''_i$ satisfies $\im L''_i=\im
  S''_i$.

  Clearly, the idempotent comonad $S_i=L_iR_i$ satisfies $\im
  L_i=\im S_i$ and restricts to $S_i'$ and induces $S''_i$.
  Hence Proposition~\ref{p:restrict-comonadic-sod} proves what we need.
\end{proof}


\section{Applications}
\label{sec:applications}
In this section, we combine the formalism of Fourier--Mukai transforms developed in
Section~\ref{sec:fourier-mukai} with the \emph{abstract version} of the conservative descent
theorem (Theorem~\ref{t:descent-comonadic-sod}) from the previous section.
This gives a \emph{geometric version} of the conservative descent theorem
(Theorem~\ref{t:main-geometric}) which is easy to apply in practice.
We illustrate the usefulness of this theorem by giving new proofs of the existence
of semi-orthogonal decompositions associated to projectivized vector bundles, blow-ups and
root stacks.

We start by reformulating our main theorem in a geometric context.
\begin{theorem}[Conservative descent]
\label{t:main-geometric}
Let $Z_1, \ldots, Z_n$ and~$X$ be algebraic stacks over some base algebraic stack $S$,
and assume that $\Phi_i \colon \D_\qc(Z_i) \to \D_\qc(X)$, for $1 \leq i \leq n$,
are Fourier--Mukai transforms over $S$.
Let $u\colon S' \to S$ be a faithfully flat morphism,
and denote the base change of the objects above by $Z'_1, \ldots, Z'_n$, $X'$ and $\Phi'_i \colon \D_\qc(Z'_i) \to \D_\qc(X')$, respectively.
Then for each $i$, the functor $\Phi_i$ is fully faithful
provided that $\Phi'_i$ is fully faithful.

Assume that all $\Phi'_i$, and therefore also all $\Phi_i$, are fully faithful.
If the sequence 
\begin{equation}
\label{eq-sod-local}
\im \Phi'_1, \ldots, \im \Phi'_n
\end{equation}
of essential images is semi-orthogonal in $\D_\qc(X')$,
then so is the sequence
\begin{equation}
\label{eq-sod}
\im \Phi_1, \ldots, \im \Phi_n
\end{equation}
in $\D_\qc(X)$.
Both sequences consist of right admissible
subcategories of $\D_\qc(X')$ and $\D_\qc(X)$, respectively.
Moreover, if \eqref{eq-sod-local} is a semi-orthogonal decomposition of $\D_\qc(X')$,
then \eqref{eq-sod} is a semi-orthogonal decomposition of $\D_\qc(X)$.
\end{theorem}
\begin{proof}
We may verify that the Fourier--Mukai transforms $\Phi_i$ are fully faithful after a
faithfully flat base change by Proposition~\ref{p:FM-functor-ff-local}.
Now assume that all our Fourier--Mukai transforms are fully faithful.
By Example~\ref{eg:FM-comonad}, we get sequences of idempotent comonads $S_i$ and $S_i'$
with the same essential images as $\Phi_i$ and~$\Phi'_i$, respectively.
Furthermore, by Proposition~\ref{p:FM-comonads-f*-compatible},
we get $g^*$-compatibilities $S_i \Ra S_i'$, where $g \colon X'
\ra X$ is the morphism induced by the base change.
By Example~\ref{eg:conservative}, the functor $g^*$ is conservative.
Hence we are in a situation where we can apply
Theorem~\ref{t-main-categorical}, 
with $g^*\colon \D_\qc(X) \to \D_\qc(X')$ playing the role of $F\colon \mathcal{T} \to \mathcal{T}'$
in the statement of the theorem.
This proves the rest of the theorem.
\end{proof}

In situations where Theorem~\ref{t:main-geometric} applies, we also get semi-orthogonal decompositions of
the categories of perfect complexes, locally bounded pseudo-coherent complexes and of the
singularity category.
Keep the notation from the statement of Theorem~\ref{t:main-geometric}.
Recall from Proposition~\ref{prop-fm-restriction} that the functors $\Phi_i$ restrict to functors
\begin{equation}
\label{eq-fm-restricted}
\Phi^\pf_i \colon \D_\pf(Z_i) \to \D_\pf(X),
\qquad
\Phi^\pc_i \colon \D^\locbd_\pc(Z_i) \to \D^\locbd_\pc(X),
\end{equation}
between categories of perfect complexes and locally bounded pseudo-coherent complexes, respectively.
We also get induced functors
\begin{equation}
\label{eq-fm-singularity}
\Phi^\sg_i \colon \D_\sg(Z_i) \to \D_\sg(X),
\end{equation}
between the singularity categories.
Note that the functors \eqref{eq-fm-restricted} and~\eqref{eq-fm-singularity},
again by Proposition~\ref{prop-fm-restriction},
have right adjoints.
Furthermore, as already discussed in Section~\ref{sec:fourier-mukai},
the functors $\Phi^\pf_i$ also admit left adjoints given by \eqref{eq-fm-pf-left}.
In particular, we get the following theorem as a direct
application of Corollary~\ref{c:restrict-sod-right-admissible}.

\begin{theorem}
\label{t:restrict-pf-coh-sg}
Let $Z_1, \ldots, Z_n$ and~$X$ be algebraic stacks over some base algebraic stack $S$,
and assume that $\Phi_i \colon \D_\qc(Z_i) \to \D_\qc(X)$, for $1 \leq i \leq n$,
are Fourier--Mukai transforms over $S$.
Consider the induced functors  \eqref{eq-fm-restricted} and~\eqref{eq-fm-singularity}.

Assume that each $\Phi_i$ is fully faithful and that
\begin{align}
\label{eq-sod-2}
  \D_\qc(X) & = \langle \im \Phi_1, \ldots, \im \Phi_n\rangle\\
  \intertext{is a semi-orthogonal decomposition. Then}
  \D_\pf(X) & = \langle \im \Phi^\pf_1, \ldots, \im \Phi^\pf_n\rangle\\
  \intertext{is a semi-orthogonal decomposition into admissible
  subcategories, and}
  \D^\locbd_\pc(X) & = \langle \im \Phi^\pc_1, \ldots, \im
  \Phi^\pc_n\rangle,\\ 
  \D_\sg(X) & = \langle \im \Phi^\sg_1, \ldots, \im \Phi^\sg_n\rangle
\end{align}
are semi-orthogonal decompositions into right admissible subcategories.
\end{theorem}

\subsection{Some auxiliary results}
Before turning to the actual applications,
we state some auxiliary results. They will help us to determine
whether a given semi-orthogonal sequence is full
(Lemma~\ref{l-sod-generator} and Lemma~\ref{lemma-compact-ample})
and to show that pull-back functors are fully faithful
(Lemma~\ref{lemma-fully-faithful}).

We start by recalling the definition of a generator for a triangulated category.
\begin{definition}
Let $\mathcal{T}$ be a triangulated category.
An object $G$ in $\mathcal{T}$ is a \define{generator} for $\mathcal{T}$
if for any object $F$ we have $F = 0$ if and only if
$\mathcal{T}(\Sigma^m G, F) = 0$ for all $m \in \mathbb{Z}$.
\end{definition}

\begin{lemma}
\label{l-sod-generator}
Let $\mathcal{T}$ be a triangulated category with a generator $G$.
Let
$\mathcal{T}_1, \ldots, \mathcal{T}_n$
be a semi-orthogonal sequence of right admissible subcategories
of $\mathcal{T}$, and let $\mathcal{T}'$ denote its triangulated hull in $\mathcal{T}$.
Then $\mathcal{T} = \mathcal{T}'$ in and only if $G \in \mathcal{T}'$.
\end{lemma}
\begin{proof}
  Trivially, $\mathcal{T}=\mathcal{T}'$ implies $G \in \mathcal{T}'$.  
  Note that $\mathcal{T}'$ is right admissible in $\mathcal{T}$ and
  that $\mathcal{T}=\langle (\mathcal{T}')^\perp, \mathcal{T}'\rangle$
  is a semi-orthogonal decomposition
  (see e.g.\ \cite[Lemma~A.9 and~A.11]{valery-olaf-matfak-semi-orth-decomp}). 
  Assume that $G$ lies in $\mathcal{T}'$.
  Then the same holds for all shifts of $G$, 
  so $(\mathcal{T}')^\perp = 0$ since $G$ is a generator. Hence
  $\mathcal{T}=\mathcal{T}'$ as desired. 
\end{proof}

For quasi-projective schemes, one can explicitly construct generators from ample line bundles.

\begin{lemma}
  \label{lemma-compact-ample}
  Let $\mathcal{L}$ be an ample line bundle on a scheme~$X$
  (see \cite[\sptag{01PS}]{stacks-project}).
  Then there exists an integer $n_0$ such that the vector bundle
  $$
  \mathcal{L}^a \oplus \mathcal{L}^{a + 1} \oplus \cdots
  \oplus \mathcal{L}^{a + n_0} 
  $$
  is a generator for $\D_\qc(X)$ for each integer $a$.
  Moreover, if there are sections $s_0, \ldots, s_n \in \HH^0(X, \mathcal{L})$
  such that the open subschemes $X_{s_i}$ are affine
  and cover $X$, then $n_0$ can be taken to be $n$.
\end{lemma}
\begin{proof}
This is a well-known fact.
The proof is almost identical to the proof of \cite[\sptag{0A9V}]{stacks-project}. 
\end{proof}

We remind the reader of the following useful criterion for fully faithfulness.
\begin{lemma}
  \label{lemma-fully-faithful}
  If $f \colon X \ra Y$ is a concentrated morphism of algebraic
  stacks, 
  then $f^* \colon \D_\qc(Y) \ra \D_\qc(X)$ is fully faithful if
  and only if the evaluation $\eta_{\mathcal{O}_Y} \colon
  \mathcal{O}_Y \ra 
  f_*f^*\mathcal{O}_Y$ of the adjunction unit $\eta$ at the structure
  sheaf is an 
  isomorphism.
\end{lemma}
\begin{proof}
  It is enough to show that the adjunction unit $\eta\colon \id
  \to f_*f^*$ is an isomorphism if $\eta_{\mathcal{O}_Y}$ is an 
  isomorphism.
  This is a direct consequence of the projection formula \cite[Corollary 4.12]{hr2017}.
  Indeed,
  for any $\mathcal{F} \in \D_\qc(Y)$,  
  the adjunction morphism $\eta_\mathcal{F} \colon \mathcal{F} \ra f_*f^*\mathcal{F}$ is in
  the obvious way identified with the upper horizontal arrow in
  the commutative diagram
  \begin{equation}
    \xymatrix{
      {\mathcal{F} \otimes \mathcal{O}_Y} 
      \ar[d]_-{\id \otimes \eta_{\mathcal{O}_Y}} 
      \ar[r]^-{\eta_{\mathcal{F} \otimes \mathcal{O}_Y}} & 
      {f_*f^*(\mathcal{F} \otimes \mathcal{O}_Y)}  
      \ar[d]^-{\sim}
      \\
      {\mathcal{F} \otimes f_*f^*\mathcal{O}_Y}
      \ar[r]^-{\sim} 
      &
      {f_*(f^*\mathcal{F} \otimes f^*\mathcal{O}_Y)}
    }
  \end{equation}
  where the lower horizontal isomorphism is the
  projection formula and the right vertical arrow is the obvious
  isomorphism. Checking that this diagram is 
  commutative essentially boils down to the 
  definition of the morphism in the projection formula.
\end{proof}

\subsection{Projectivized vector bundles}
\label{sec:projectivized-bundles}
As our first application of conservative descent, we  generalize the semi-orthogonal decomposition
associated to a projectivized vector bundle, first established by
Be\u{\i}linson \cite{beilinson1978}
in the case of a projective space over a field and by Orlov
\cite[Lemma~2.5, Theorem~2.6]{orlov1992} in the relative setting.

\begin{theorem}
\label{t:sod-projective-bundle}
Let $S$ be an algebraic stack and let $\pi\colon \mathbb{P}(\mathcal{E}) \to S$
be the projectivization of a finite locally free $\mathcal{O}_S$-module
$\mathcal{E}$ of rank $n+1$.
Then the functors
$$
\Phi_i \colon
\D_\qc(S) \ra \D_\qc(\mathbb{P}(\mathcal{E})), \qquad \mathcal{F} \mapsto
\mathcal{O}_{\mathbb{P}(\mathcal{E})}(i) \otimes \pi^*(\mathcal{F}),
$$
are fully faithful Fourier--Mukai transforms.
For any  integer $a$,
we have a semi-orthogonal decomposition
\begin{equation}
\label{eq-pb-sod}
\D_\qc(\mathbb{P}(\mathcal{E})) = \langle \im \Phi_a, \ldots, \im \Phi_{a + n}\rangle
\end{equation}
into right admissible subcategories. 
\end{theorem}
\begin{proof}
By Proposition~\ref{prop-preserving}, each $\Phi_i$ is a Fourier--Mukai transform over $S$
with $K = \mathbb{P}(\mathcal{E})$,
$\mathcal{K} = \mathcal{O}_{\mathbb{P}(\mathcal{E})}(i)$,
$p = \pi$,
$q = \id$
in the notation from Definition~\ref{def-fourier-mukai}.
Moreover, base change along any morphism is compatible with
taking the projectivization of a vector bundle (cf.\
\cite[\sptag{01O3}]{stacks-project}).
Hence the theorem can be verified after a faithfully flat base
change, by the geometric conservative descent Theorem~\ref{t:main-geometric}.
This reduces the problem to the situation where $S = \Spec R$ is an affine scheme
and $\mathcal{E} = \mathcal{O}_S^{n+1}$ is a trivial vector bundle.

First we verify that $\Phi_i$ is fully faithful.
Since twisting with a line bundle induces an equivalence of categories,
it is enough to show that $\pi^*$ is fully faithful.
By Lemma~\ref{lemma-fully-faithful}, it is enough to show that the canonical
morphism $\mathcal{O}_S \to \pi_*\pi^*\mathcal{O}_S$ is an isomorphism.
Since $S$ is affine, 
this follows from the well known fact that $\HH^q(\mathbb{P}^n_R, \mathcal{O}_{\mathbb{P}^n_R})$
is equal to $R$ for $q=0$ and vanishes otherwise
(see e.g.~\cite[\sptag{01XT}]{stacks-project}).

Next we show that the sequence of categories in \eqref{eq-pb-sod}
is semi-orthogonal. 
Let $\mathcal{F}$ and $\mathcal{G}$ be objects of $\D_\qc(S)$.
  \begin{align}
    \label{eq-bundle-so}
    \Hom(\Phi_i(\mathcal{F}), \Phi_j(\mathcal{G}))
    & \cong
      \Hom(\mathcal{F},
      \pi_\ast(\mathcal{O}_{\mathbb{P}(\mathcal{E})}(j - i)
      \otimes \pi^\ast\mathcal{G})) \\ 
    & \cong
      \Hom(\mathcal{F},\pi_\ast
      (\mathcal{O}_{\mathbb{P}(\mathcal{E})}(j - i) ) \otimes
      \mathcal{G}), \notag
  \end{align}
  where the first bijection follows from the adjunction and twisting with
  $\mathcal{O}_{\mathbb{P}(\mathcal{E})}(-i)$,
  and the second bijection follows from the projection formula.
  But the sheaf cohomology of $\mathcal{O}_{\mathbb{P}(\mathcal{E})}(j - i)$
  vanishes whenever $-n \leq j - i < 0$, by 
  \cite[\sptag{01XT}]{stacks-project}, so 
  $\pi_\ast (\mathcal{O}_{\mathbb{P}(\mathcal{E})}(j -
  i))=0$.
  Hence $\eqref{eq-bundle-so}$ vanishes for $-n \leq j - i < 0$,
  so the sequence in \eqref{eq-pb-sod} is indeed semi-orthogonal.

  Finally, we show that the sequence is
  full.
  By  Lemma~\ref{lemma-compact-ample},
  the object
  $\mathcal{O}_{\mathbb{P}(\mathcal{E})}(a) \oplus \cdots \oplus \mathcal{O}_{\mathbb{P}(\mathcal{E})}(a + n)$
  is a generator for $\D_\qc(\mathbb{P}(\mathcal{E}))$,
  so by Lemma~\ref{l-sod-generator} it is enough to verify that this
  object is in the triangulated hull
  of the semi-orthogonal sequence $\im \Phi_a, \ldots,
  \im \Phi_{a + n}$ of right admissible subcategories.
  But this is obvious since
  $\Phi_i(\mathcal{O}_S) = \mathcal{O}_{\mathbb{P}(\mathcal{E})}(i)$.
\end{proof}

By Theorem~\ref{t:restrict-pf-coh-sg},
we immediately get the following corollary of Theorem~\ref{t:sod-projective-bundle}.

\begin{corollary}
\label{c:sod-projective-bundle}
Keep the notation from Theorem~\ref{t:sod-projective-bundle}.
Similarly as in the statement of Theorem~\ref{t:restrict-pf-coh-sg},
we let $\Phi^\pf_i$, $\Phi^\pc_i$ and~$\Phi^\sg_i$ denote the induced functors
between derived categories of perfect complexes,
derived categories of locally bounded pseudo-coherent complexes,
and singularity categories, respectively.

Then we have a semi-orthogonal decomposition
\begin{align}
\D_\pf(\mathbb{P}(\mathcal{E})) & = \langle \im \Phi^\pf_a, \ldots, \im \Phi^\pf_{a + n}\rangle\\
\intertext{into admissible subcategories and semi-orthogonal decompositions}
\D^\locbd_\pc(\mathbb{P}(\mathcal{E})) & = \langle \im \Phi^\pc_a, \ldots, \im \Phi^\pc_{a + n}\rangle,\\
\D_\sg(\mathbb{P}(\mathcal{E})) & = \langle \im \Phi^\sg_a, \ldots, \im \Phi^\sg_{a + n}\rangle
\end{align}
into right admissible subcategories.
\end{corollary}

\subsection{Blow-ups}
\label{sec:blow-ups}
In this subsection, we describe the semi-orthogonal decomposition associated to the blow-up of an algebraic stack
in a regular sheaf of ideals.
This semi-orthogonal decomposition was first described by Orlov \cite[Theorem~4.3]{orlov1992}
in the less general setting of a blow-up of a smooth variety in a smooth subvariety.

The standard reference for blow-ups in the generality we work in
is \cite[Exposé~VII]{sga6}.
Recall from Remark~\ref{rem-regular-immersion}
that we follow the conventions about regular immersions from this source. 
Let $\iota\colon Z \to X$ be a regular closed immersion of algebraic stacks.
By blowing up $X$ in $Z$, we obtain a cartesian diagram
\begin{equation}
\label{eq-blow-up}
\xymatrix{
E \ar[d]_-\rho \ar[r]^-\kappa & \widetilde{X} \ar[d]^-\pi \\
Z \ar[r]^-\iota & X \\
}
\end{equation}
where $E$ is the exceptional divisor of the blow-up.
The morphism $\pi$ is projective with twisting sheaf $\mathcal{O}_{\widetilde{X}}(1) \cong \mathcal{O}_{\widetilde{X}}(-E)$.
Recall that the conormal bundle $\mathcal{N}_{Z/X}$ of the regular immersion
is locally free and that $\rho \colon E \to Z$ is isomorphic to the projectivization of $\mathcal{N}_{Z/X}$.
If the rank of $\mathcal{N}_{Z/X}$ is constant $c$,
then we say that $Z$ has constant codimension $c$ in $X$.

\begin{theorem}
\label{t:sod-blow-up}
Let $X$ be an algebraic stack and $\iota\colon Z \hookrightarrow
X$ a regular closed immersion of constant
codimension $c \geq 0$.
Consider the blow-up diagram \eqref{eq-blow-up},
and define the functors 
\begin{align*}
\mathrlap{\Phi_i}{\phantom{\Phi_0}}  & \colon \D_\qc(Z) \, \to \D_\qc(\widetilde{X}), &
\mathcal{F} & \mapsto \mathcal{O}_{\widetilde{X}}(-iE) \otimes \kappa_* \rho^*(\mathcal{F}),
\qquad i < 0,\\
\Phi_0  & \colon \D_\qc(X) \, \to \D_\qc(\widetilde{X}), &
\mathcal{F} & \mapsto \pi^*(\mathcal{F}).
\end{align*}
Then each functor $\Phi_i$ is a Fourier--Mukai transform over
$X$.
Moreover, for each $i \in \{-c+1, \ldots, 0\}$,
the functor $\Phi_i$ is fully faithful and we have a semi-orthogonal decomposition
\begin{equation}
\label{eq-sod-blow-up}
\D_\qc(\widetilde{X})=\langle \im \Phi_{-c + 1}, \ldots, \im \Phi_{0} \rangle
\end{equation}
into right admissible subcategories. 
\end{theorem}
\begin{proof}
By Proposition~\ref{prop-preserving} and the projection formula, each $\Phi_i$ is a
Fourier--Mukai transform over $X$.
Indeed, for $i < 0$, we let
$K = E$,
$\mathcal{K} = \mathcal{O}_E(i) \cong \kappa^*\mathcal{O}_{\widetilde{X}}(i)$,
$p = \rho$,
$q = \kappa$
in the notation from Definition~\ref{def-fourier-mukai}.
Similarly, for $i = 0$, we let
$K = \widetilde{X}$,
$\mathcal{K} = \mathcal{O}_{\widetilde{X}}$,
$p = \pi$,
$q = \id$.
Hence by conservative descent, as stated in
Theorem~\ref{t:main-geometric}, 
and the fact that base change along any flat morphism preserves
blow-ups (cf.\ \cite[\sptag{0805}]{stacks-project}),
the theorem can be verified after a faithfully flat base change.
In particular, we may assume that $X = \Spec R$ is an affine scheme and that the ideal
defining $Z$ is generated by a regular sequence of length $c$.
The theorem is trivial when $c \leq 1$, so let us assume that $c > 1$.

First we prove that $\Phi_i$ is fully faithful for each $i$.
For $i = 0$ it is enough to verify that the canonical morphism
$\mathcal{O}_X \to \pi_*\pi^*\mathcal{O}_X$ is an isomorphism, by
Lemma~\ref{lemma-fully-faithful}.
This follows from the fact that $\HH^r(\widetilde{X}, \mathcal{O}_{\widetilde{X}})$
vanishes for $r > 0$ and is isomorphic to $R$ for $r = 0$, as is
shown in \cite[Exposé VII, Lemme~3.5]{sga6}.
Assume that $i < 0$ and let $\mathcal{E}$ and $\mathcal{F}$ be objects of $\D_\qc(Z)$.
The map
$\Hom(\mathcal{E}, \mathcal{F})
\to
\Hom(\Phi_i(\mathcal{E}), \Phi_i(\mathcal{F}))$
factors as
\begin{align}
\Hom(\mathcal{E}, \mathcal{F})
& \xrightarrow{\sim}
\Hom(\rho^\ast\mathcal{E}, \rho^\ast\mathcal{F})
\xrightarrow{\epsilon}
\Hom(\kappa^\ast\kappa_\ast\rho^\ast\mathcal{E},
  \rho^\ast\mathcal{F}) \\
\notag
& \xrightarrow{\sim}
\Hom(\kappa_\ast\rho^\ast\mathcal{E}, \kappa_\ast\rho^\ast\mathcal{F})
\xrightarrow{\sim}
\Hom(\Phi_i(\mathcal{E}), \Phi_i(\mathcal{F})),
\end{align}
where the first map is bijective by
Theorem~\ref{t:sod-projective-bundle}, and the last two maps are the obvious bijections.
The map $\epsilon$ is obtained
from the evaluation
$\kappa^\ast\kappa_\ast \rho^\ast \mathcal{E} \to \rho^\ast
\mathcal{E}$
of the adjunction counit at $\rho^\ast \mathcal{E}$.
It suffices to prove that $\epsilon$ is an isomorphism.
Since $E$ is an effective Cartier divisor,
the adjunction counit evaluated at an arbitrary object
$\mathcal{M}$ of $\D_\qc(E)$ fits into a triangle
\begin{equation}
\label{eq-tria-reg-immersion}
\Sigma \mathcal{M}(1)
\ra \kappa^* \kappa_* \mathcal{M}
\ra \mathcal{M} 
\ra \Sigma^2 \mathcal{M}(1)
\end{equation}
in $\D_\qc(E)$ by \cite[Lemma~4.2]{bls2016}
or \cite[Porisme 3.5]{thomason1993}, where $\mathcal{M}(1)$ denotes the usual Serre twist
$\mathcal{M} \otimes\mathcal{O}_{\widetilde{X}}(1)$.
Consider the particular case $\mathcal{M} = \rho^\ast\mathcal{E}$.
Note that $E$ is the projectivization of the conormal bundle of the inclusion $Z \subset X$,
which has rank $c$ by assumption.
By the semi-orthogonal decomposition in Theorem~\ref{t:sod-projective-bundle},
the functor $\Hom(-, \rho^\ast\mathcal{F})$
vanishes on the first and the last object in \eqref{eq-tria-reg-immersion}.
Hence $\epsilon$ is an isomorphism as required.

Next we verify that the sequence 
of categories
in \eqref{eq-sod-blow-up}
is semi-orthogonal.
First assume that $-c + 1 \leq j < i < 0$ and let $\mathcal{E}$ and~$\mathcal{F}$
be objects in~$\D_\qc(Z)$.
Using the adjunction isomorphism and twisting with $\mathcal{O}_{\widetilde{X}}(-i)$,
we get
\begin{equation}
\label{eq-so-blow-up}
\Hom(\Phi_i(\mathcal{E}), \Phi_j(\mathcal{F}))
\cong
\Hom(\kappa^\ast\kappa_\ast\rho^\ast\mathcal{E}, (\rho^\ast\mathcal{F})(j-i)).
\end{equation}
Similarly as above, we apply the functor
$\Hom(-, (\rho^\ast\mathcal{F})(j-i))$
to the triangle \eqref{eq-tria-reg-immersion} with $\mathcal{M} = \rho^\ast\mathcal{E}$.
By the semi-orthogonal decomposition in
Theorem~\ref{t:sod-projective-bundle}
it vanishes on the first and third object because
$- c + 1 < j - i < 0$.
This implies that both sides of \eqref{eq-so-blow-up} also vanish, as desired.
Now assume instead that $\mathcal{E}$ is an object of $\D_\qc(X)$,
that $\mathcal{F}$ is an object of $\D_\qc(Z)$, and that $-c+1 \leq
i <0$.
Then
\begin{align}
\Hom(\Phi_0(\mathcal{E}), \Phi_i(\mathcal{F}))
& \cong
\Hom(\kappa^\ast\pi^\ast\mathcal{E}, (\rho^\ast\mathcal{F})(i)) \\
\notag
& \cong
\Hom(\rho^\ast\iota^\ast\mathcal{E}, (\rho^\ast\mathcal{F})(i))
\end{align}
vanishes, again by the semi-orthogonal decomposition in
Theorem~\ref{t:sod-projective-bundle}.
Hence the sequence in \eqref{eq-sod-blow-up} is indeed semi-orthogonal.

Finally, we prove that our sequence is full. 
Let $\mathcal{T}$ denote the triangulated hull in
$\D_\qc(\tilde{X})$ of the subcategories
$\im \Phi_{-c + 1}, \ldots, \im \Phi_{0}$.
It clearly contains $\mathcal{O}_{\widetilde{X}}$ and $\mathcal{O}_E(i)$
for $-c < i < 0$.
By twisting the exact sequence
\begin{equation}
  0 \to \mathcal{O}_{\widetilde{X}}(1) \to
  \mathcal{O}_{\widetilde{X}} \to \mathcal{O}_E \to 0 
\end{equation}
with $\mathcal{O}_{\widetilde{X}}(i)$ for $-c < i < 0$,
we see that $\mathcal{T}$ also contains $\mathcal{O}_{\widetilde{X}}(i)$
for $-c < i < 0$.
In particular, the category $\mathcal{T}$ contains
$\mathcal{G} = \mathcal{O}_{\widetilde{X}}(-c+1) \oplus \cdots \oplus \mathcal{O}_{\widetilde{X}}(0)$.
Since the ideal defining $Z$ is generated by $c$ elements,
the blow-up $\widetilde{X}$ embeds into $\mathbb{P}^{c-1}_R$.
Hence the bundle $\mathcal{G}$ generates $\D_\qc(\widetilde{X})$ by Lemma~\ref{lemma-compact-ample}.
Therefore $\mathcal{T} = \D_\qc(\widetilde{X})$ by
Lemma~\ref{l-sod-generator}. 
This concludes the proof. 
\end{proof}

By Theorem~\ref{t:restrict-pf-coh-sg},
we immediately get the following corollary of Theorem~\ref{t:sod-blow-up}.

\begin{corollary}
\label{c:sod-blow-up}
Keep the notation from Theorem~\ref{t:sod-blow-up}.
Similarly as in the statement of Theorem~\ref{t:restrict-pf-coh-sg},
we let $\Phi^\pf_i$, $\Phi^\pc_i$ and~$\Phi^\sg_i$ denote the induced functors
between the derived categories of perfect complexes,
derived categories of locally bounded pseudo-coherent complexes,
and singularity categories, respectively.

Then we have a semi-orthogonal decomposition
\begin{align}
\D_\pf(\widetilde{X}) & = \langle \im \Phi^\pf_{-c+1}, \ldots, \im \Phi^\pf_{0}\rangle\\
\intertext{into admissible subcategories and semi-orthogonal decompositions}
\D^\locbd_\pc(\widetilde{X}) & = \langle \im \Phi^\pc_{-c+1}, \ldots, \im \Phi^\pc_{0}\rangle, \\ 
\D_\sg(\widetilde{X}) & = \langle \im \Phi^\sg_{-c+1}, \ldots, \im \Phi^\sg_{0}\rangle
\end{align}
into right admissible subcategories.
\end{corollary}

\subsection{Root stacks}
The \emph{root construction} is a construction which formally adjoins a root of some order $r > 0$
to an effective Cartier divisor on a scheme or an algebraic stack.
This is a purely stacky construction,
which yields a genuine algebraic stack in all but the trivial cases where we take the first root of a divisor.
The root construction was originally described by Cadman in~\cite{cadman2007}.
We refer to \cite[Section~3]{bls2016} for a summary of its most important properties.

A root stack is a birational modification which has many similarities with a blow-up.
In particular, its derived category admits a semi-orthogonal decomposition,
as first noted by Ishii--Ueda \cite[Proposition~6.1]{iu2015}.
We gave a proof in a more general setting in \cite[Theorem~4.7]{bls2016}.
In the proof we left out the details of the reduction to the local setting
with a reference to this article.
We restate the theorem here as Theorem~\ref{t:sod-root-stack} and fill in the missing part of the proof.

Let $X$ be an algebraic stack and let $\iota\colon D \hookrightarrow X$ be an effective Cartier divisor.
For a given integer $r > 0$, the root construction gives a diagram
\begin{equation}
\label{eq-root-diagram}
\xymatrix{
E \ar[d]_-{\rho} \ar[r]^-{\kappa} & \widetilde{X} \ar[d]^-{\pi} \\
D \ar[r]^-{\iota} & X. \\
}
\end{equation}
Note that this diagram fails to be cartesian whenever $r > 1$.
Rather the pull-back of $D$ along $\pi$ can be identified with $rE$.
This is the motivation for the term \emph{root construction}.

\begin{theorem}
\label{t:sod-root-stack}
Let $X$ be an algebraic stack and $\iota\colon D \hookrightarrow X$ an effective Cartier divisor.
Fix an integer $r > 0$ and consider the root diagram \eqref{eq-root-diagram}
associated to the $r$-th root construction.
Define the functors
\begin{align*}
\mathrlap{\Phi_i}{\phantom{\Phi_0}}  & \colon \D_\qc(D) \, \to \D_\qc(\widetilde{X}), &
\mathcal{F} & \mapsto \mathcal{O}_{\widetilde{X}}(-iE) \otimes \kappa_* \rho^*(\mathcal{F}),
\qquad -r < i < 0,\\
\Phi_0  & \colon \D_\qc(X) \, \to \D_\qc(\widetilde{X}), &
\mathcal{F} & \mapsto \pi^*(\mathcal{F}).
\end{align*}
Then each functor $\Phi_i$ is a Fourier--Mukai transform over $X$.
Moreover, for each $i \in \{-r + 1, \dots, 0\}$,
the functor $\Phi_i$ is fully faithful and we have a semi-orthogonal decomposition
\begin{equation}
\label{eq-sod-root}
\D_\qc(\widetilde{X})=\langle \im \Phi_{-r + 1}, \ldots, \im \Phi_{0} \rangle
\end{equation}
into right admissible subcategories. 
\end{theorem}
\begin{proof}
By Proposition~\ref{prop-preserving} and the projection formula, each $\Phi_i$ is a Fourier--Mukai transform over $X$.
Indeed, for $i \in \{-r + 1, \dots, -1\}$, we let
$K = E$,
$\mathcal{K} = \kappa^*\mathcal{O}_{\widetilde{X}}(-iE)$,
$p = \rho$,
$q = \kappa$
in the notation from Definition~\ref{def-fourier-mukai}.
For $i = 0$, we let
$K = \widetilde{X}$,
$\mathcal{K} = \mathcal{O}_{\widetilde{X}}$,
$p = \pi$,
$q = \id$.
Hence by conservative descent, as stated in Theorem~\ref{t:main-geometric},
the theorem can be verified after a faithfully flat base change.
In particular, we may assume that $X = \Spec R$ is an affine scheme and that the ideal
defining $D$ is generated by a single regular
element.

The rest of the proof, which is very similar to the proof of Theorem~\ref{t:sod-blow-up},
is written down in full detail in~\cite[Theorem~4.7]{bls2016}.
The semi-orthogonal decomposition in the statement of the cited
theorem refers to the category $\D_\pf(\widetilde{X})$, but the proof applies equally well to $\D_\qc(\widetilde{X})$.
Note that there is an obvious typographical error in the formulation of the cited theorem;
two occurrences of $\D(\widetilde{X})$ should be replaced by $\D_\pf(\widetilde{X})$.
\end{proof}

By Theorem~\ref{t:restrict-pf-coh-sg},
we immediately get the following corollary of Theorem~\ref{t:sod-root-stack}.

\begin{corollary}
\label{c:sod-root-stack}
Keep the notation from Theorem~\ref{t:sod-root-stack}.
Similarly as in the statement of Theorem~\ref{t:restrict-pf-coh-sg},
we let $\Phi^\pf_i$, $\Phi^\pc_i$ and~$\Phi^\sg_i$ denote the induced functors
between the derived categories of perfect complexes,
derived categories of locally bounded pseudo-coherent complexes,
and singularity categories, respectively.

Then we have a semi-orthogonal decomposition
\begin{align}
\D_\pf(\widetilde{X}) & = \langle \im \Phi^\pf_{-r+1}, \ldots, \im \Phi^\pf_{0}\rangle
\intertext{into admissible subcategories and semi-orthogonal decompositions}
\D^\locbd_\pc(\widetilde{X}) & = \langle \im \Phi^\pc_{-r+1}, \ldots, \im \Phi^\pc_{0}\rangle,
\\ 
\D_\sg(\widetilde{X}) & = \langle \im \Phi^\sg_{-r+1}, \ldots, \im \Phi^\sg_{0}\rangle
\end{align}
into right admissible subcategories.
\end{corollary}


\bibliographystyle{myalpha}
\bibliography{references}

\end{document}